\documentclass{article}
\usepackage[a4paper, total={5.5in, 8in}]{geometry}
\usepackage[english]{babel}
\usepackage[T1]{fontenc}
\usepackage[utf8]{inputenc}
\usepackage{amsmath}
\usepackage{amsthm}
\usepackage{amssymb}
\usepackage{natbib}
\usepackage[title]{appendix}
\usepackage{algpseudocode}
\usepackage{accents}
\usepackage{mathrsfs}
\usepackage{graphicx,mathtools,tikz,hyperref}
\usepackage{caption}
\usepackage{subcaption}
\usepackage{algorithm}
\usepackage{algpseudocode}
\usepackage{stackengine}

\usepackage{amsmath,amsfonts,bm}

\def\eqref#1{Eq.~\ref{#1}}

\def\peqref#1{\ref{#1}}

\def\1{\bm{1}}

\def\vc{{\bm{c}}}

\def\vh{{\bm{h}}}

\def\vp{{\bm{p}}}
\def\vq{{\bm{q}}}

\def\vw{{\bm{w}}}
\def\vx{{\bm{x}}}
\def\vy{{\bm{y}}}
\def\vz{{\bm{z}}}

\def\mA{{\bm{A}}}
\def\mB{{\bm{B}}}
\def\mC{{\bm{C}}}

\def\mH{{\bm{H}}}

\def\mN{{\bm{N}}}

\def\mQ{{\bm{Q}}}

\DeclareMathAlphabet{\mathsfit}{\encodingdefault}{\sfdefault}{m}{sl}
\SetMathAlphabet{\mathsfit}{bold}{\encodingdefault}{\sfdefault}{bx}{n}

\def\gI{{\mathcal{I}}}

\def\sI{{\mathbb{I}}}

\DeclareMathOperator*{\argmax}{arg\,max}
\DeclareMathOperator*{\argmin}{arg\,min}

\newcommand{\cO}{\mathcal{O}}

\newcommand{\RR}{\mathbb{R}}

\newcommand{\la}{\langle}
\newcommand{\ra}{\rangle}

\newcommand{\norm}[1]{\left\lVert#1\right\rVert}
\numberwithin{equation}{section}
\usetikzlibrary{positioning}
\PassOptionsToPackage{hyphens}{url}\usepackage{hyperref}

\theoremstyle{plain}
\newtheorem{theorem}{Theorem}[section]
\newtheorem{proposition}[theorem]{Proposition}
\newtheorem{lemma}[theorem]{Lemma}

\newtheorem{corollary}[theorem]{Corollary}
\theoremstyle{definition}
\newtheorem{example}{Example}[section]

\newtheorem{assumption}[theorem]{Assumption}

\theoremstyle{remark}
\newtheorem{remark}[theorem]{Remark}

\renewcommand{\d}[1]{\ensuremath{\operatorname{d}\!{#1}}}
\usepackage{enumitem}
\setlist[itemize]{itemsep=0pt, topsep=0pt}
\setlist[enumerate]{itemsep=-2pt, topsep=0pt}

\providecommand{\keywords}[1]
{
  \small	
  \textbf{\textit{Keywords---}} #1
}

\title{Primal-Dual Damping algorithms for optimization}

\author{Xinzhe Zuo$^{\star}$,  Stanley Osher$^{\star}$,  Wuchen Li$^{\dagger}$\ \ \\
\textsuperscript{$\star$} Department of Mathematics, University of California, Los Angles\\
\texttt{\{zxz, sjo\}@math.ucla.edu}\\
\textsuperscript{$\dagger$} Department of Mathematics, University of South Carolina  \\
 \texttt{wuchen@mailbox.sc.edu}\\
}

\date{}

\begin{document}

\maketitle

\begin{abstract}  
    We propose an unconstrained optimization method based on the well-known primal-dual hybrid gradient (PDHG) algorithm. We first formulate the optimality condition of the unconstrained optimization problem as a saddle point problem. We then compute the minimizer by applying generalized primal-dual hybrid gradient algorithms. Theoretically, we demonstrate the continuous-time limit of the proposed algorithm forms a class of second-order differential equations, which contains and extends the heavy ball ODEs and Hessian-driven damping dynamics. Following the Lyapunov analysis of the ODE system, we prove the linear convergence of the algorithm for strongly convex functions. Experimentally, we showcase the advantage of algorithms on several convex and non-convex optimization problems by comparing the performance with other well-known algorithms, such as Nesterov's accelerated gradient methods. In particular, we demonstrate that our algorithm is efficient in training two-layer and convolution neural networks in supervised learning problems. 
\end{abstract}\hspace{10pt}

\section{Introduction}
 Optimization is one of the essential building blocks in many applications, including scientific computing and machine learning problems. One of the classical algorithms for unconstrained optimization problems is the gradient descent method, which updates the state variable in the negative gradient direction at each step \citep{boyd2004convex}. Nowadays, accelerated gradient descent methods have been widely studied. Typical examples include Nesterov's accelerated gradient method \citep{nesterov1983method}, Polyak's heavy ball method \citep{polyak1964some}, and Hessian-driven damping methods \citep{chen2019first,attouch2019fast,attouch2020first,attouch2021convergence}.

On the other hand, some first-order methods are introduced to solve linear-constrained optimization problems. Typical examples include the primal-dual hybrid gradient (PDHG) method \citep{chambolle2011first} and the alternating direction method of multipliers (ADMM) \citep{boyd2011distributed,GABAY197617}. They are designed to solve an inf-sup saddle point type problem, which updates the gradient descent direction for the minimization variable and applies the gradient ascent direction for the maximization variable. Both PDHG and ADMM are designed for solving optimization problems with affine constraints. \citet{ouyang2015accelerated} proposed accelerated linearized ADMM, which incorporates a multi-step acceleration scheme into linearized ADMM. 
Recently, the PDHG method has been extended into solving nonlinear-constrained minimization problems \citep{valkonen2014primal}. 


In this paper, we study a general class of accelerated first-order methods for unconstrained optimization problems. We reformulate the original optimization problem into an inf-sup type saddle point problem, whose saddle point solves the optimality condition. We then apply a linearized preconditioned primal-dual hybrid gradient algorithm to compute the proposed saddle point problem. 

The main description of the algorithm is as follows. Consider the following inf-sup problem for a $\mathcal{C}^2$ strongly convex function $f$ over $\RR^d$ \begin{equation}\label{eq:pdd_formulation_regu}
    \inf_{\vx\in\mathbb{\RR}^d} \sup_{\vp\in\mathbb{\RR}^d} \quad \la \nabla f(\vx), \vp \ra - \frac{\varepsilon}{2}\norm{\vp}^2 \,,
\end{equation}
where $\vp$ is a constructed ``dual variable'', $\varepsilon>0$ is a constant, $\la\cdot,\cdot\ra$ is an Euclidean inner product, and $\|\cdot\|$ is an Euclidean norm. 
We later prove that the solution to the saddle point problem \peqref{eq:pdd_formulation_regu} gives the global minimum of $f$. We propose a linearized preconditioned PDHG algorithm for solving the above inf-sup problem: 
\begin{subequations}\label{eq:discrete_linearized_pdhg}
\begin{align}
    \vp^{n+1} &= \vp^n + \sigma \mA(\vx^n) \nabla f(\vx^n)-\sigma \varepsilon \mA(\vx^n) \vp^{n+1} \,,\\
    \tilde{\vp}^{n+1} &= \vp^{n+1} + \omega(\vp^{n+1}-\vp^n)\,, \\
    \vx^{n+1} &= \vx^n - \tau \mC(\vx^n) \tilde{\vp}^{n+1}\,,
\end{align}
\end{subequations}
where $n=1, 2,\cdots$ is the iteration step, $\tau$, $\sigma>0$ are stepsizes for the updates of $\vx$, $\vp$,  respectively, and $\omega>0$ is a parameter. In the above algorithm, $\mC(\vx^n) = \mB(\vx^n) \nabla^2 f(\vx^n)$, where $\mA(\vx^n)\in\mathbb{\RR}^{d\times d}$, and $\mB(\vx^n)\in\mathbb{\RR}^{d\times d}$ act as preconditioners on the updates of $\vp^{n+1}$ and $\vx^{n+1}$, respectively. This paper will only focus on the simple case where $\mA(\vx) = A \sI$ for some constant $A>0$. Although there is a second-order term $\nabla^2 f(\vx^n)$ in the update of $\vx^{n+1}$ (hidden in $\mC(\vx^n)$), our algorithm is still a first-order algorithm by choosing $\mB(\vx^n) \nabla^2 f(\vx^n) = \mC(\vx^n)$ for some $\mC$ that is easy to compute. For example, we test that $\mC = \sI$ is a very good choice in most of our numerical examples. See empirical choices of parameters in our numerical sections. 

Our method forms a class of ordinary differential equation systems
in terms of $(\vx, \vp)$ in the continuous limit $\tau$, $\sigma\rightarrow 0$. We call it the primal-dual damping (PDD) dynamics.  We show that the PDD dynamics form a class of second-order ODEs, which contains and extends the inertia Hessian-driven damping dynamics \citep{chen2019first,attouch2019fast}. Theoretically, we analyze the convergence property of PDD dynamics. If $f$ is a quadratic function of $\vx$, with constant $\mA$, $\mB$, the PDD dynamic satisfies a linear ODE system. Under suitable choices of parameters, we obtain a similar convergence acceleration in heavy ball ODE \citep{siegel2019accelerated}. Moreover, for general nonlinear function $f$, we have the following informal theorem characterizing the convergence speed of our algorithm: 
\begin{theorem}[Informal]\label{thm:informal}
    Let $f:\RR^d \to \RR $ be a $\mathcal{C}^4$ strongly convex function. Let $\vx^*$ be the global minimum of $f$ and $\vp^* = 0$. 
    Then, the iteration $(\vx^n,\vp^n)$ produced by \peqref{eq:discrete_linearized_pdhg} converges to the saddle point $(\vx^*,\vp^*)$ if $\tau$, $\sigma$, are small enough. Moreover, 
{
\begin{equation*}
  \|\vp^n\|^2+\|\nabla f(\vx^n)\|^2\leq (\|\vp^0\|^2+\|\nabla f(\vx^0)\|^2) (1- \frac{\mu^2}{M+\delta})^n,    
\end{equation*}
where $\mu=\min_{\vx}\lambda_{\min}(\nabla^2 f(\vx)\mC(\vx))$, $\mC(\vx) = \mB(\vx) \nabla^2 f(\vx)$, $\delta>0$ depends on the initial condition, 
and $M>0$ depends on $\mC(\vx)^T\big(\nabla^3 f(\vx) \nabla f(\vx)+(\nabla^2 f(\vx))^2\big)\mC(\vx)$, $\tau$, $\sigma$, $A$, $\varepsilon$, and $\omega$. The detailed version is given in Theorem \ref{thm:restatement}.
}
\end{theorem}
Numerically, we test the algorithms in both convex and non-convex optimization problems. In convex optimization, we demonstrate the fast convergence results of the proposed algorithm with selected preconditioners, compared with the gradient descent method, Nesterov accelerated gradient method, and Heavy ball damping method. This justifies the convergence analysis. We also test our algorithm for several well-known non-convex optimization problems. Some examples, such as the Rosenbrock and Ackley functions, demonstrate the potential advantage of our algorithms in converging to the global minimizer. In particular, we compare our algorithms with the stochastic gradient descent method, Adam, for training two-layer and convolutional neural network functions in supervised learning problems. This showcases the potential advantage of the proposed methods in terms of convergence speed and test accuracy.

PDHG has been widely used in linear-constrained optimization problems \citep{chambolle2011first}. Recently, \citet{valkonen2014primal} applied the PDHG for nonlinearly constrained optimization problems. They proved the asymptotic convergence for the nonlinear coupling saddle point problems. It is different from our PDHG algorithm for computing unconstrained optimizations. And we show the linear convergence for a particular nonlinear coupling saddle point problem. Meanwhile, Nesterov accelerated gradient methods and Hessian damping algorithms can also be formulated in both discrete-time updates and continuous-time second-order ODEs.  \citet{wibisono2016variational} also introduced the idea of Bregman Lagrangian to study a family of accelerated methods in continuous time limit. It forms a nonlinear second-order ODE. 
Compared to them, the PDD algorithm induces a generalized second-order ODE system, which contains both heavy ball ODE \citep{siegel2019accelerated} and Hessian damping dynamics \citep{chen2019first,attouch2019fast,attouch2020first,attouch2021convergence}. 
For example, when $C=\sI$, algorithm  \eqref{eq:discrete_linearized_pdhg} can be viewed as the other time discretization of Hessian damping dynamics \citep{chen2019first,attouch2019fast}. It provides a different update in discrete time update. We only evaluate the gradient of $f$ once, whereas Attouch's algorithm \citep{attouch2020first} evaluates the gradient of $f$ twice. In numerical experiments, we demonstrate that the proposed algorithm outperforms Nesterov accelerated methods and Hessian-driven damping methods in some non-convex optimization problems, including supervised learning problems for training neural network functions. 

Our work is also related to preconditioning, an important technique in numerical linear algebra \citep{trefethen2022numerical} and numerical PDEs \citep{rees2010preconditioning,park2021preconditioned}. In general, preconditioning aims to reduce the condition number of some operators to improve convergence speed. One famous example would be preconditioning gradient descent by the inverse of the Hessian matrix, which gives rise to Newton's method. In recent years, preconditioning techniques have also been developed in training neural networks \cite{osher2022laplacian,kingma2014adam}. Adam \citep{kingma2014adam} is arguably one of the most popular optimizers in training deep neural networks. It can also be viewed as a preconditioned algorithm using a diagonal preconditioner that approximates the diagonal of the Fisher information matrix \citep{pascanu2013revisiting}. Shortly after \citet{chambolle2011first} developed PDHG for constrained optimization, the same authors also studied preconditioned PDHG method \citep{chambolle2011diagonal}, in which they developed a simple diagonal preconditioner that can guarantee convergence without the need to compute step sizes. \citet{liu2021acceleration} proposed non-diagonal preconditioners for PDHG and showed close connections between preconditioned PDHG and ADMM. \citet{park2021preconditioned} studied the preconditioned Nesterov's accelerated gradient method and proved convergence in the induced norm. \citet{jacobs2019solving} introduced a preconditioned norm in the primal update of the PDHG method and improved the step size restriction of the PDHG method. 

Our paper is organized as follows. In Section \ref{section:pdd_for_optimization} we provide some background and derivations of our algorithm. We also provide the ODE formulations for our primal-dual damping dynamics. In Section \ref{section:Lyapunov_analysis}, we prove our main convergence results for the algorithm. In Section \ref{section:numerics} we showcase the advantage of our algorithm through several convex and non-convex examples. In particular, we show that our algorithm can train neural networks and is competitive with commonly used optimizers, such as SGD with Nesterov's momentum and Adam. We conclude in Section \ref{section:discussion} with more discussions and future directions.  

\section{Primal-dual damping algorithms for optimizations}\label{section:pdd_for_optimization}
In this section, we first review PDHG algorithms for constrained optimization problems. We then construct a saddle point problem for the unconstrained optimization problem and apply the preconditioned PDHG algorithm to compute the proposed saddle point problem. We last derive an ODE system, which takes the limit of stepsizes in the PDHG algorithm. It forms a second-order ODE, which generalizes the Hessian-driven damping dynamics. We analyze the convergence properties of the ODE system for quadratic optimization problems.     

\subsection{Review PDHG for constrained optimization}

In \citet{chambolle2011first}, the following saddle point problem was considered: 
\begin{equation}
    \min_{x\in X} \max_{y\in Y} \la Kx,y\ra + G(x) - F^*(y) \,,\label{eq:saddle_point_problem}
\end{equation}
where $X$ and $Y$ are two finite-dimensional real vector spaces equipped with inner product $\la\cdot,\cdot \ra$ and norm $\|\cdot\| = \la\cdot,\cdot \ra^{1/2}$. The map $K:X\to Y$ is a continuous linear operator. $G:X\to [0,+ \infty]$ and $F^*:Y\to [0,+\infty]$ are proper, convex, lower semi-continuous (l.s.c.) functions. $F^*$ is the convex conjugate of a convex l.s.c. function $F$. It is straightforward to verify that \peqref{eq:saddle_point_problem} is the primal-dual formulation of the nonlinear primal problem 
$$ \min_{x\in X} F(Kx) + G(x)\,. 
$$
Then the PDHG algorithm for saddle point problem \peqref{eq:saddle_point_problem} is given by 
\begin{subequations}
\begin{align}
    y^{n+1} &= (I + \sigma \partial F^*)^{-1}(y^n + \sigma K \tilde{x}^n)\,, \\
    x^{n+1} &= (I + \tau \partial G)^{-1}(x^n + \tau K^* y^{n+1})\,, \\
    \tilde{x}^{n+1} &= x^{n+1} + \omega(x^{n+1}-x^n)\,,
\end{align}\label{eq:original_PDHG}
\end{subequations}
where $I$ is the identity operator and $(I + \sigma \partial F)^{-1}$ is the resolvent operator, which is defined the same way as the proximal operator
\begin{align*}
     (I+\tau \partial F)^{-1}(y) &= \argmin_x \frac{\|x-y\|^2}{2\tau} + F(x) \\
    &= \textrm{prox}_{\tau F}(y) 
\end{align*}
When $\omega = 1$, \citet{chambolle2011first} proved convergence if $\tau\sigma \|K\|^2<1$, where $\|\cdot\|$ is the induced operator norm. It is worth noting that the convergence analysis requires that $K$ is a linear operator. 
\subsection{Saddle point problem for unconstrained optimization}
We consider the problem of minimizing a $\mathcal{C}^2$ strongly convex function $f:\RR^d \to \RR$ over $\RR^d$. Instead of directly solving for $\nabla f(\vx^*) = 0$, we consider the following saddle point problem: 
\begin{equation}\label{eq:pdd_formulation}
    \inf_{\vx \in \RR^d} \sup_{\vp\in\RR^d} \quad\la \nabla f(\vx), \vp \ra \,,
\end{equation}
due to the following proposition. 
\begin{proposition}\label{prop:saddle_optimality}
    Let $f:\RR^d \to \RR$ be a $\mathcal{C}^2$ strongly convex function. Then the saddle point to \peqref{eq:pdd_formulation} is the unique global minimum of $f$. 
\end{proposition}
\begin{proof}
    Directly differentiating \peqref{eq:pdd_formulation} and setting the derivatives to 0 yields 
    \begin{align*}
        \nabla f(\vx^*) &= 0 \,,\\
        \nabla^2 f(\vx^*) \vp^* &= 0 \,.
    \end{align*}
By the strong convexity of $f$, we obtain that $\vx^*$ is the unique global minimum and $\vp^* = 0$. 
\end{proof}
 Recall that $\vp^*=0$ by the optimality condition. Thus we make the following change to our saddle point formulation.
We add a regularization term in \peqref{eq:pdd_formulation}:
\begin{equation}\label{eq:saddle_regu}
    \inf_{\vx\in \RR^d} \sup_{\vp\in \RR^d} \quad \la \nabla f(\vx), \vp \ra - \frac{\varepsilon}{2}\norm{\vp}^2 \,,
\end{equation}
where $\varepsilon>0$ is a constant. This regularization term further drives $\vp$ to $0$. Similar to Proposition \ref{prop:saddle_optimality}, we have the following proposition
\begin{proposition}
    Let $f:\RR^d \to \RR$ be a $\mathcal{C}^2$ strongly convex function. Then the saddle point to \peqref{eq:saddle_regu} is the unique global minimum of $f$. 
\end{proposition}
\begin{proof}
    Directly differentiating \peqref{eq:saddle_regu} and setting derivatives to 0 yields 
    \begin{align*}
        \nabla f(\vx^*) &= \varepsilon \vp^*\, ,\\
        \nabla^2 f(\vx^*) \vp^* &= 0\, .
    \end{align*}
    Since $f$ is strongly convex, we have $\nabla^2 f(\vx^*) \succ 0$ and the second equation implies $\vp^* = 0$. Then the first equation implies $\nabla f(\vx^*) =0$. Since $f$ is strongly convex, we conclude that $\vx^*$ is the unique global minimum.
\end{proof}
\subsection{PDHG for unconstrained optimization}
We apply the scheme given by \peqref{eq:original_PDHG} to the saddle point problem \peqref{eq:saddle_regu} (set $F=G=0$ and identify $K\vx = \nabla f(\vx)$ in \peqref{eq:saddle_point_problem}). Thus,  
\begin{subequations}\label{eq:pdd_pdhg}
\begin{align}
    \vp^{n+1} &= \argmax_{\vp} \quad \la \nabla f(\vx^n), \vp \ra - \frac{\varepsilon}{2}\norm{\vp}^2 - \frac{\norm{\vp - \vp^n}^2_{\mA(\vx^n)^{-1}}}{2\sigma}\,, \\
     \widetilde{\vp}^{n+1} &= \omega(\vp^{n+1} - \vp^n) + \vp^{n+1}\,, \label{eq:pdd_pdhg_2}\\
    \vx^{n+1} &=\argmin_{\vx} \quad \la \nabla f(\vx),\widetilde{\vp}^{n+1} \ra + \frac{\norm{\vx - \vx^n}^2_{\mB(\vx^n)^{-1}}}{2\tau}\,,\label{eq:pdd_pdhg_3}
\end{align}
\end{subequations}
where we have added symmetric positive definite matrices $\mA(\vx^n)$, $\mB(\vx^n)\in\mathbb{R}^{d\times d}$, as preconditioners for updates of $\vp$, $\vx$, respectively. We also denote the norm $\norm{\vh}^2_{\mA^{-1}}$ as 
$\vh^T\mA^{-1}\vh$, where $\vh\in\mathbb{R}^d$. 

As mentioned, the convergence analysis of PDHG relies on the assumption that $K$ is a linear operator. So we can not apply the same convergence analysis to \peqref{eq:pdd_pdhg} since $\nabla f(\vx)$ is not necessarily linear in $\vx$. By taking the optimality conditions of \peqref{eq:pdd_pdhg}, we find that $\vp^{n+1}$ and $\vx^{n+1}$ solves 
\begin{subequations} \label{eq:pdd_pdhg_system}
\begin{align}
    \vp^{n+1} + \sigma\varepsilon  \mA(\vx^n) \vp^{n+1} -\vp^n - \sigma \mA(\vx^n) \nabla f(\vx^n) &= 0\,, \label{eq:p_updates} \\
    \tau \mB(\vx^n) \nabla^2f(\vx^{n+1}) \big[(1+\omega)\big(\sI + \sigma \varepsilon \mA(\vx^n)\big)^{-1}\sigma \mA(\vx^n) \nabla f(\vx^n) \nonumber \\ 
    +\big((\omega+1)\big(\sI + \sigma \varepsilon \mA(\vx^n)\big)^{-1}-\omega \sI \big) \vp^n \big] + (\vx^{n+1} - \vx^n) &= 0\ ,  \label{eq:x_updates_2}
\end{align}
\end{subequations}
where we substitute the update \eqref{eq:pdd_pdhg_2} into update \eqref{eq:x_updates_2}. We use $\sI$ to represent the identity matrix in \eqref{eq:x_updates_2}. 

Note that the update for $\vx^{n+1}$ in \eqref{eq:x_updates_2} is implicit, unless $\nabla^2 f(\vx)$ does not depend on $\vx$. We also remark that the update for $\vx^{n+1}$ in \eqref{eq:x_updates_2} will be explicit if we perform a gradient step instead of a proximal step in \eqref{eq:pdd_pdhg_3}. To be more precise, when $\mB = \sI$, the linearized version of \eqref{eq:pdd_pdhg_3} can be written as 
$$ \vx^{n+1} = \mathrm{prox}_{\tau \la \nabla f(\cdot),\widetilde{\vp}^{n+1} \ra} (\vx^n)\,. 
$$
Taking a gradient step instead of proximal step yields
\begin{equation}\label{eq:pdd_pdhg_3_explicit}
    \vx^{n+1} = \vx^n - \tau  \nabla^2 f(\vx^n)\widetilde{\vp}^{n+1} 
\end{equation}
For general choice of preconditioner $\mB(\vx^n)$,  the linearized version of \eqref{eq:pdd_pdhg_3} satisfies
 \begin{equation*}
    \vx^{n+1} = \vx^n - \tau  \mB(x^n)\nabla^2 f(\vx^n)\widetilde{\vp}^{n+1}= \vx^n - \tau \mC(x^n)\widetilde{\vp}^{n+1}. 
\end{equation*}
Here we always denote a matrix function $\mC$, such that
\begin{equation*}
    \mC(\vx^n):=\mB(\vx^n)\nabla^2f(\vx^n). 
\end{equation*}
For simplicity of presentation, we only consider the simple case where $\mA(\vx^n) = A \sI $ for some $A>0$. We now summarize the linearized update \eqref{eq:pdd_pdhg_system} into the following algorithm. 
\begin{algorithm}
\caption{Linearized Primal-Dual Damping Algorithm}\label{alg:pdd}
\begin{algorithmic}
\Require Initial guesses $\vx^0\in \mathbb{R}^d$, $\vp^0\in\mathbb{R}^d$; Stepsizes $\tau>0$, $\sigma>0$; Parameters $A>0$, $\varepsilon>0$, $\omega>0$, $\mC\succ 0$.  
\While{$n=1,2,\cdots,$ not converge}
\State $\vp^{n+1} = \frac{1}{1+\sigma \varepsilon A}\vp^n + \frac{\sigma A}{1+\sigma \varepsilon A}\nabla f(\vx^n);$
 \State   $\tilde{\vp}^{n+1} = \vp^{n+1} + \omega(\vp^{n+1}-\vp^n);$
 \State   $\vx^{n+1} = \vx^n - \tau \mC(\vx^n) \tilde{\vp}^{n+1};$ 
\EndWhile
\end{algorithmic}
\end{algorithm}

We note that Algorithm \ref{alg:pdd} and update \eqref{eq:pdd_pdhg_system} are different methods for solving saddle point problem \eqref{eq:pdd_formulation}. In this paper, we focus on the computation and analysis of Algorithm \ref{alg:pdd}. 
\subsection{PDD dynamics}

An approach for analyzing optimization algorithms is by first studying the continuous limit of the algorithm using ODEs \citep{su2015differential,siegel2019accelerated,attouch2019fast}. The advantage of doing so is that ODEs provide insights into the convergence property of the algorithm. 

We first reformulate the proposed algorithm \eqref{eq:pdd_pdhg_system} into a first-order ODE system. 
\begin{proposition}
    As $\tau,\sigma\to 0$ and $\sigma\omega \to \gamma$, both updates in \peqref{eq:pdd_pdhg_system} and Algorithm \ref{alg:pdd} can be formulated as a discrete-time update of the following ODE system. 
\begin{subequations}\label{eq:p_x_updates_ode_regu}
\begin{align}
    \Dot{\vp} &= \mA(x) \nabla f(\vx) - \varepsilon \mA(x)\vp \,,\label{eq:p_updates_ode_regu} \\
    \Dot{\vx} &= -\mC(\vx)(\vp + \gamma (\mA(x) \nabla f(\vx) - \varepsilon \mA(x)\vp)) \,,\label{eq:x_updates_ode_regu}
\end{align}
\end{subequations}
 where $\mC(\vx)=\mB(\vx)\nabla^2f(\vx)$ and the initial condition satisfies $\vx(0)=\vx^0$, $\vp(0)=\vp^0$. {Suppose that $\nabla f$ is Lispchitz continuous and each index in matrix $\mA$, $\mC$ is continuous and bounded. Then, there exists a unique solution for the ODE system \eqref{eq:p_x_updates_ode_regu}.}  A stationary state $(\vx^*, \vp^*)$ of ODE system \eqref{eq:p_x_updates_ode_regu} satisfies 
\begin{equation*}
\nabla f(\vx^*)=0, \quad \vp^*=0.
\end{equation*}
\end{proposition}
\begin{proof}
    Rearranging \eqref{eq:p_updates} and \eqref{eq:x_updates_2}, we have
    \begin{align*}
        \frac{\vp^{n+1}-\vp^n}{\sigma} &= \mA(\vx^n) \nabla f(\vx^n) - \varepsilon  \mA(\vx^n) \vp^{n+1} \,,\\
          \frac{\vx^{n+1} - \vx^n}{\tau} &=  -\mB(\vx^n) \nabla^2f(\vx^{n+1}) \big[(1+\omega)\big(\sI + \sigma \varepsilon \mA(\vx^n)\big)^{-1}\sigma \mA(\vx^n) \nabla f(\vx^n) \nonumber \\ 
    &\qquad +\big((\omega+1)\big(\sI + \sigma \varepsilon \mA(\vx^n)\big)^{-1}-\omega \sI \big) \vp^n \big]\,.
    \end{align*}
Taking the limit as $\tau,\sigma \to 0$ and $\sigma \omega \to \gamma$, we obtain 
\begin{align*}
    \Dot{\vp} &= \mA(x)\nabla f(\vx) - \varepsilon \mA(x)  \vp\,,  \\
    \Dot{\vx} &= -\mB(\vx)\nabla^2 f(\vx)(\vp + \gamma (\mA(x) \nabla f(\vx) - \varepsilon \mA(x)\vp)) \,.
\end{align*}
Similarly, the update in Algorithm \ref{alg:pdd} also converges to the ODE system \eqref{eq:p_x_updates_ode_regu}. Clearly, a stationary state satisfies $\vp^*=0$, $\nabla f(\vx^*)=0$.
\end{proof}


\begin{proposition}[Primal-dual damping second order ODE]\label{prop:pdd_ode_regu}
   The ODE system \eqref{eq:p_x_updates_ode_regu} satisfies the following second-order ODE 
\begin{equation}\label{eq:time_dependent_general_formulation}
    \ddot{\vx} +\big[ \varepsilon \mA  + \gamma \mC \mA\nabla^2 f(\vx)-\dot{\mC} \mC^{-1}\big] \dot{\vx} +\mC\mA \nabla f(\vx)=0\,.
\end{equation}
Here $\dot \mC=\frac{d}{dt}\mC(\vx(t))$. 
\end{proposition}
The proof follows by direct calculations and can be found in Appendix \ref{appendix:proof_pdd_ode_regu}. We note that the formulation given by \eqref{eq:time_dependent_general_formulation} includes several important special cases in the literature. In a word, we view \eqref{prop:pdd_ode_regu} as a preconditioned accelerated gradient flow. 

\begin{example}
Let $\mC=\mA=\mathbb{I}$ and $\gamma\neq 0$. Then equation \eqref{prop:pdd_ode_regu} satisfies 
\begin{equation}\label{eq:Hessian_driven_damping_general}
      \ddot{\vx} +\epsilon \dot \vx  + \gamma \nabla^2 f(\vx) \dot{\vx} + \nabla f(\vx)=0\,,
\end{equation}
which is an inertial system with Hessian-driven damping \citep{attouch2020first}.
\end{example}
\begin{remark}
    In the case of $\mC =\mA= \sI$, although the derived second order ODE \eqref{eq:time_dependent_general_formulation} is the same as the one in \citet{attouch2020first} at a continuous time level, our algorithm \ref{alg:pdd} provides a different time discretization from the one in \citet{attouch2020first}. 
\end{remark}
\begin{example}
   Let $\mC=\mA = \sI$, $\gamma(t)=0$. Then equation \eqref{prop:pdd_ode_regu} satisfies the heavy ball ODE \citep{siegel2019accelerated} 
    \begin{equation} \label{eq:hb_ode}
    \ddot{\vx} + \varepsilon \dot{\vx} + \nabla f(\vx) = 0\,.
    \end{equation}
\end{example}
\begin{example}
    Let $\mC=\mA = \sI$,  $\gamma(t)=0$, $\varepsilon(t) = \frac{3}{t}$. Then equation \eqref{prop:pdd_ode_regu} satisfies the Nesterov ODE \citep{su2015differential}: 
    \begin{equation} \label{eq:nesterov_ode}
    \ddot{\vx} + \frac{3}{t} \dot{\vx} + \nabla f(\vx) = 0\,.
    \end{equation}
\end{example}

We next provide a convergence analysis of ODE \eqref{eq:p_x_updates_ode_regu} for quadratic optimization problems. We demonstrate the importance of preconditioners in characterizing the convergence speed of ODE \eqref{eq:p_x_updates_ode_regu}.

\begin{theorem}\label{prop:convergence_rate_original_pdd}
Suppose $f(\vx) = \frac{1}{2}\vx^T \mQ \vx$ for some symmetric positive definite matrix $\mQ\in \mathbb{R}^{d\times d}$. Assume $\mA$, $\mB$ are constant matrices. In this case, equation \eqref{eq:p_x_updates_ode_regu} satisfies the linear ODE system: 
\begin{equation*}
    \begin{pmatrix}
        \dot{\vx} \\ \dot{\vp} 
    \end{pmatrix} = \begin{pmatrix}
        -\gamma \mB \mQ\mA\mQ &-\mB\mQ(\sI-\gamma \varepsilon \mA ) \\
         \mA\mQ & -\varepsilon \mA 
    \end{pmatrix} \begin{pmatrix}
        \vx \\ \vp 
    \end{pmatrix}\,.
\end{equation*}
Suppose that $\mA$ commutes with $\mQ$, such that $\mA\mQ=\mQ\mA$. Suppose $\mA$ and $\mB\mQ\mA\mQ$ are simultaneously diagonalizable and have positive eigenvalues. Let $\mu_1\geq \ldots \geq \mu_n >0$ be the eigenvalues of $\mB\mQ\mA\mQ$ and $a_i$ the $i$-th eigenvalue of $\mA$ (not necessarily in descending order) in the same basis.  Then 
\begin{enumerate}
    \item[(a)] The solution of ODE system \peqref{eq:p_x_updates_ode_regu} converges to $(\vx^*,\vp^*) = (0,0)$: 
$$
\|(\vx(t),\vp(t))\| \leq  \|(\vx_0,\vp_0)\| \exp(\alpha t)\,,
$$
where 
$$
\alpha = \max_i \frac{1}{2}\big[-\gamma \mu_i - \varepsilon a_i + \Re\big(\sqrt{(\gamma\mu_i + \varepsilon)^2-4\mu_i} \big)   \big]\,.
$$
\item[(b)] When $\mA = \sI, \varepsilon=0$, the optimal convergence rate is achieved at $\gamma^* = \frac{2\sqrt{\mu_1}}{\sqrt{\mu_n(2\mu_1-\mu_n)}} $. The corresponding rate is $\alpha = \frac{-\sqrt{\mu_n}}{\sqrt{2-\frac{1}{\kappa}}}$, where $\kappa = \mu_1/\mu_n >1$. 
\item[(c)] Moreover, when $\gamma=\varepsilon = 0$, the system will not converge for any initial data $(\vx_0,\vp_0) \neq (0,0)$. 
\item[(d)] If $\mA = \sI$, $\gamma \leq \frac{1}{\sqrt{\mu_1}}$, $\varepsilon = 2\sqrt{\mu'}-\gamma \mu'$ for some $\mu' \leq \mu_n$, then 
$$
\alpha = -\sqrt{\mu'}-\frac{\gamma}{2}(\mu_n - \mu')\leq -\sqrt{\mu'} \,. 
$$
\end{enumerate}
\end{theorem}
We defer the proof to Appendix \ref{appendix:proof_proposition_rate}.
\begin{remark}
    If $\omega$ is bounded, then we have $\gamma = \cO(\sigma)$. Then, in the limit as $\sigma \to 0$, we also have that $\gamma \to 0$. By Theorem \ref{prop:convergence_rate_original_pdd} (c), the ODE system \peqref{eq:p_x_updates_ode_regu} does not converge for any initial data. 
\end{remark}
\begin{remark}
    If $\mu'$ is an estimate of the smallest eigenvalue $\mu_n$, then the convergence speed for the solution of heavy ball ODE is $\exp(-\sqrt{\mu'}t)$. In Theorem \ref{prop:convergence_rate_original_pdd} (d), if $\gamma=0$ and $\mu' = \mu_n$, then $\alpha = -\sqrt{\mu_n}$ which is the same as the convergence rate of the heavy ball ODE \citep{siegel2019accelerated}. However, if $\gamma >0$ and $\mu' < \mu_n$, then we have $\alpha = -\sqrt{\mu'}-\gamma(\mu_n - \mu') < -\sqrt{\mu'} $, which converges faster than the heavy ball ODE.  
\end{remark}

\section{Lyapunov Analysis}\label{section:Lyapunov_analysis}
In this section, we present the main theoretical result of this paper. We provide the convergence analysis for general objective functions in both continuous-time ODEs \eqref{eq:p_x_updates_ode_regu} and discrete-time Algorithm \ref{alg:pdd}. From now on, we make the following two assumptions for the convergence analysis. 
\begin{assumption}\label{assumption:mu_L}
 There exists two constants $L\geq \mu>0$ such that $\mu \sI \preceq \mC_0(\vx) \preceq L \sI$ for all $\vx$, where $\mC_0(\vx) = \nabla^2 f(\vx) \mB(\vx) \nabla^2 f(\vx) $, and $\mu \leq 1$. 
\end{assumption}
\begin{assumption}\label{assumption:L'}
     There exists a constant $L'>0$ such that 
\begin{equation}\label{condition:L'}
    \mC(\vx)^T\big(\nabla^3 f(\vx) \nabla f(\vx)+(\nabla^2 f(\vx))^2\big)\mC(\vx) \preceq L' \sI 
\end{equation}
for all $\vx$, where $\mC(\vx) = \mB(\vx) \nabla^2  f(\vx)$. 
\end{assumption}

\subsection{Continuous time Lyapunov analysis}
In this subsection, we establish convergence results of the ODE system \eqref{eq:p_x_updates_ode_regu}. 

\begin{theorem}\label{thm:continuous_Lyapunov_eps}
Consider the ODE system \eqref{eq:p_x_updates_ode_regu} with an initial condition $(x(0), p(0))\in\mathbb{R}^{2d}$. Define the functional
\begin{equation}\label{eq:Lyapunov_2}
    \gI(\vx,\vp) = \frac{1}{2}(\|\vp\|^2 + \|\nabla f (\vx)\|^2) \,.
\end{equation}
Suppose Assumption \ref{assumption:mu_L} holds, we have 
\begin{equation}
    \gI(\vx(t),\vp(t)) \leq \gI(\vx(0),\vp(0)) \exp(-2\lambda t)\,,
\end{equation}
where
\begin{align}
    \lambda = \min\Big\{&\mu\gamma A - \frac{1}{2}|A - \mu(1 - \varepsilon \gamma A)|, L\gamma A - \frac{1}{2}|A - L(1 - \varepsilon \gamma A)|, \nonumber \\
    & \varepsilon A - \frac{1}{2}|A - \mu(1 - \varepsilon \gamma A)|,   \varepsilon A - \frac{1}{2}|A - L(1 - \varepsilon \gamma A)| \Big\} \nonumber
\end{align}
In particular, when $\gamma = \frac{1}{\mu}, \varepsilon = 1, A =  \frac{\mu+L}{2+(\mu+L)\varepsilon\gamma}$, then $\lambda = \frac{\mu}{2}$.   
\end{theorem}
\begin{proof}
It is straightforward to compute the following
\begin{align}
    \frac{\d \gI}{\d t} &= \la \vp,\dot{\vp} \ra + \la \nabla f, \nabla^2 f \dot{\vx} \ra  \nonumber \\
    &= -\nabla f^T \mC_0 \gamma \mA \nabla f - \vp^T\varepsilon \mA \vp + \nabla f^T \big(\mA  -\mC_0(\sI-\varepsilon \gamma \mA )\big)\vp 
\end{align}
We shall find $\lambda$ such that $\frac{\d \gI}{\d t} + 2\lambda \gI \leq 0$. Then we obtain the exponential convergence by Gronwall's inequality, i.e., 
\begin{equation*}
    \gI(\vx(t),\vp(t)) \leq \gI(\vx(0),\vp(0)) \exp(-2\lambda t)\ .
\end{equation*}
We can compute 
\begin{align}\label{eq:Lyapunov_continuous_derivative}
    \frac{\d \gI}{\d t} + 2\lambda \gI &= \nabla f^T \big(-\mC_0 \gamma \mA +\lambda \sI \big)\nabla f  + \vp^T\big(-\varepsilon \mA   + \lambda \sI \big) \vp \nonumber \\
    &\qquad + \nabla f^T \big(\mA  -\mC_0(\sI-\varepsilon \gamma \mA ) \big)\vp \,.
\end{align}
By Lemma \ref{lemma:matrix_CS}, we obtain
the following sufficient conditions for $\frac{\d \gI}{\d t} + 2\lambda \gI \leq 0$
\begin{subequations}\label{eq:conditions_on_lambda_epsilon}
\begin{align}
     -\varepsilon A  + \lambda + \frac{1}{2}|\xi_i(1-\varepsilon\gamma A ) -A| &\leq 0  \\
   \lambda -\xi_i \gamma A + \frac{1}{2}| \xi_i(1-\varepsilon\gamma A  ) -A| &\leq 0
\end{align}
\end{subequations}
where $\xi_i(\vx)$ is the eigenvalue of $\mC_0(\vx)$. By our assumptions, we have $L\geq \xi_1(\vx) \geq \ldots \geq \xi_n(\vx) \geq \mu$. \eqref{eq:conditions_on_lambda_epsilon} give two upper bounds on $\lambda$. Define $g_1(\xi) = \varepsilon A + \frac{1}{2}| \xi(1-\varepsilon\gamma A  ) -A|$, and $g_2(\xi) = \xi \gamma A - \frac{1}{2}| \xi(1-\varepsilon\gamma A  ) -A| $ on the interval $[\mu,L]$. Then \eqref{eq:conditions_on_lambda_epsilon} implies that 
\begin{equation}
    \lambda \leq g_j(\xi_i) \,,
\end{equation}
for all $i=1,\ldots,n$ and $j=1,2$. Since each $g_j(\xi)$ is a piece-wise linear in $\xi$, it is not hard to see that 
\begin{align*}
    \min_{\xi \in[\mu,L]} g_j(\xi) = \min \{ g_j(\mu),g_j(L) \} \,,
\end{align*}
for $j=1,2$. This proves the formula for $\lambda$. When $A =\frac{\mu+L}{2+(\mu+L)\varepsilon\gamma} $, we have $g_1(\mu) = g_1(L)$, and  
$$\mu(1-\varepsilon\gamma A ) -A = - L(1-\varepsilon\gamma A ) + A\,.
$$
Further, requiring $g_1(\mu) = g_2(\mu)$ yields $\varepsilon = \mu\gamma$. And we obtain 
\begin{align}
    \lambda &= \mu \gamma A - \frac{1}{2}|A - \mu(1 - \varepsilon \gamma \mA)|\nonumber \\
    &= \mu \gamma A - \frac{1}{2}(A - \mu(1 - \varepsilon \gamma A))\nonumber \\
    &= \frac{\mu}{2}+ A(\gamma \mu-\frac{1}{2}\gamma^2\mu^2-\frac{1}{2}) \nonumber \\
    &= \frac{\mu}{2} - \frac{A}{2}(\gamma\mu-1)^2.
\end{align}
We note that $\lambda$ is maximized when taking $\gamma=\mu^{-1}$. We obtain $\lambda =\frac{\mu}{2}$. 

\end{proof}

\subsection{Discrete time Lyapunov analysis}
In this subsection, we study the convergence criterion for the discretized linearized PDHG flow given by \eqref{eq:discrete_linearized_pdhg} and Algorithm \ref{alg:pdd}.

From now on, we assume that $f$ is a $\mathcal{C}^4$ strongly convex function.

We can rewrite the iterations as 
\begin{subequations}
\begin{align}
    \vp^{n+1} &= \frac{1}{1+\sigma \varepsilon A}\vp^n + \frac{\sigma A}{1+\sigma \varepsilon A}\nabla f(\vx^n)\,,\\
    \vx^{n+1} &= \vx^n - \tau \mB(\vx^n) \nabla^2 f(\vx^n)\left(\frac{1-\varepsilon \gamma A}{1+\sigma \varepsilon A}\vp^n + \frac{\sigma A + \gamma A}{1+\sigma \varepsilon A} \nabla f(\vx^n) \right)\,,
\end{align}
\end{subequations}
where $\gamma = \sigma \omega$. We define the following notations which will be used later. 
\begin{equation}
\mN(\vx^n) = \frac{1}{1+\sigma \varepsilon A}\begin{pmatrix}
\mB(\vx^n) \nabla^2 f(\vx^n)(\sigma A+\gamma A) & \mB(\vx^n) \nabla^2 f(\vx^n)(1-\varepsilon\gamma A)\\
-\frac{\sigma}{\tau}A & \frac{\sigma}{\tau}\varepsilon A
\end{pmatrix}\,.
\end{equation}
And 
$$\mH(\vx^n)= \rm{sym} \begin{pmatrix}
\begin{pmatrix}
    \nabla^2 f(\vx^n) & 0 \\
    0 &\sI 
    \end{pmatrix} \cdot \mN(\vx^n)
\end{pmatrix}\,.$$ 
\begin{remark}
    The matrix $\mN(\vx^n)$ and $\mH(\vx^n)$ also depends on the $\tau$, $\sigma$, $A$, $\varepsilon$ and $\omega$. 
\end{remark}
Define the Lyapunov functional in discrete time as 
$$ \gI(\vx^n,\vp^n) = \frac{1}{2}\|\nabla f(\vx^n) \|^2 + \frac{1}{2} \|\vp^n\|^2\,. 
$$

\begin{theorem}\label{thm:discrete_exp_decay_1}
Suppose that there exists positive constants $\lambda, M_1 \in \RR_+$, such that 
\begin{align*}
    \mH(\vx) &\succeq \lambda \sI\,,\\
    \mN(\vx)^T \nabla^2 \gI(\tilde{\vx},\tilde{\vp})
    \mN(\vx) &\preceq M_1 \sI,
\end{align*}
for all $\vx,\tilde{\vx}\in \RR^n$. If $\tau = a \frac{\lambda}{M}$ for some $a \in (0,2)$, then the functional $\gI(\vx^n,\vp^n)$ decreases geometrically, i.e. 
$$\gI(\vx^{n},\vp^{n}) \leq \gI(\vx^{0},\vp^{0})\big(1+(a^2-2a)\frac{\lambda^2}{2M_1}\big)^n\,.
$$
\end{theorem}
\begin{proof}
It follows from our definition of $\mN(\vx^n)$ that 
\begin{equation}\label{eq:xp_N_gradf_p}
    \begin{pmatrix}
    \vx^{n+1} - \vx^n \\
    \vp^{n+1} - \vp^n
    \end{pmatrix} = -\tau \mN(\vx^n) \begin{pmatrix}
    \nabla f(\vx^n) \\
    \vp^n
    \end{pmatrix} \,,
\end{equation}

By the mean-value theorem, we obtain 
\begin{align}
    &\gI(\vx^{n+1},\vp^{n+1})-\gI(\vx^{n},\vp^{n})\nonumber \\
    = &\begin{pmatrix}
    \nabla_{\vx}\gI(\vx^n,\vp^n)\\
    \nabla_{\vp}\gI(\vx^n,\vp^n)
    \end{pmatrix}^T 
    \begin{pmatrix}
    \vx^{n+1}-\vx^n \\
    \vp^{n+1} - \vp^n
    \end{pmatrix} + \frac{1}{2}\begin{pmatrix}
    \vx^{n+1}-\vx^n \\
    \vp^{n+1} - \vp^n
    \end{pmatrix}^T  \nabla^2 \gI(\tilde{\vx},\tilde{\vp})\begin{pmatrix}
    \vx^{n+1}-\vx^n \\
    \vp^{n+1} - \vp^n 
    \end{pmatrix} \nonumber 
\end{align}
where $(\tilde{\vx},\tilde{\vp})$ is in between $(\vx^{n+1},\vp^{n+1})$ and $(\vx^n,\vp^n)$. And 
\begin{align*}
    \nabla_{\vx}\gI(\vx^n,\vp^n) &= \nabla^2 f(\vx^n) \nabla f(\vx^n) \,,\\
    \nabla_{\vp}\gI(\vx^n,\vp^n) &= \vp^n\,, \\
    \nabla^2 \gI(\vx^n,\vp^n)&= \begin{pmatrix}
    \nabla^3 f(\vx^n) \nabla f(\vx^n)+\nabla^2 f(\vx^n)\nabla^2 f(\vx^n) & 0 \\
    0& \sI
    \end{pmatrix}\,.
\end{align*}
Then using \eqref{eq:xp_N_gradf_p} and definition of $\mH(\vx^n)$, we obtain 
\begin{align}\label{eq:I_second_taylor_inexact}
    &\gI(\vx^{n+1},\vp^{n+1})-\gI(\vx^{n},\vp^{n})\nonumber \\
    =&-\tau \begin{pmatrix}
    \nabla f(\vx^n) \\
    \vp^n
    \end{pmatrix}^T \begin{pmatrix}
    \nabla^2 f(\vx^n) &0 \\
    0 &\sI 
    \end{pmatrix} \cdot \mN(\vx^n) \begin{pmatrix}
    \nabla f(\vx^n) \\
    \vp^n
    \end{pmatrix}\nonumber \\
    &+\frac{\tau^2}{2}\begin{pmatrix}
    \nabla f(\vx^n) \\
    \vp^n
    \end{pmatrix}^T
    \mN(\vx^n)^T \nabla^2 \gI(\tilde{\vx},\tilde{\vp})
    \mN(\vx^n)
    \begin{pmatrix}
    \nabla f(\vx^n) \\
    \vp^n
    \end{pmatrix}\nonumber \\
    =&-\tau \begin{pmatrix}
    \nabla f(\vx^n) \\
    \vp^n
    \end{pmatrix}^T \mH(\vx^n) \begin{pmatrix}
    \nabla f(\vx^n) \\
    \vp^n
    \end{pmatrix}\nonumber \\
    &+ \frac{\tau^2}{2}\begin{pmatrix}
    \nabla f(\vx^n) \\
    \vp^n
    \end{pmatrix}^T
    \mN(\vx^n)^T \nabla^2 \gI(\tilde{\vx},\tilde{\vp})
    \mN(\vx^n)
    \begin{pmatrix}
    \nabla f(\vx^n) \\
    \vp^n
    \end{pmatrix}\,,
\end{align}
From \eqref{eq:I_second_taylor_inexact} and our assumption on $\mN(\vx)$ and $\mH(\vx)$, we obtain 
\begin{align}
    \gI(\vx^{n+1},\vp^{n+1})-\gI(\vx^{n},\vp^{n})&\leq \big(-\tau \lambda + \frac{\tau^2 M_1}{2}\big) \gI(\vx^{n},\vp^{n}) \nonumber \\
    &= \frac{M_1}{2}\big((\tau-\frac{\lambda}{M_1})^2 - \frac{\lambda^2}{M_1^2}\big) \gI(\vx^{n},\vp^{n})\nonumber \\
    &=(a^2-2a)\frac{\lambda^2}{2M_1}\gI(\vx^{n},\vp^{n})\,,
\end{align}
where we used $\tau = a \frac{\lambda}{M_1}$. Hence,
$$\gI(\vx^{n+1},\vp^{n+1}) \leq \gI(\vx^{n},\vp^{n})\big(1+(a^2-2a)\frac{\lambda^2}{2M_1}\big) \leq \gI(\vx^{0},\vp^{0})\big(1+(a^2-2a)\frac{\lambda^2}{2M_1}\big)^{n+1}\,.
$$
When $0<a<2$, we have $a^2-2a<0$. Thus we obtain the desired convergence result.
\end{proof}

\begin{theorem}\label{thm:discrete_exp_decay_2}
Let $f:\RR^d \to \RR$ be a $\mathcal{C}^4$ strongly convex function. Suppose $(\vx^0,\vp^0)$ satisfies 
\begin{equation}\label{condition:x0p0}
     \gI(\vx^0,\vp^0)^{1/2} \leq \frac{\delta}{\tau D_0 \|\mN(\vx)\|_2^3 }  \,,
\end{equation}
for some $\delta>0$ and all $\vx$. Here 
$$ D_0  = \sup_{\vx,\vp,\vx',\vp'} \frac{ \begin{pmatrix}
    \vx' \\
    \vp'
    \end{pmatrix}^T  \left( \nabla^3 \gI(\vx,\vp) \begin{pmatrix}
    \vx' \\
    \vp'
    \end{pmatrix}\right)  \begin{pmatrix}
    \vx' \\
    \vp'
    \end{pmatrix}}{\left\| \begin{pmatrix}
    \vx' \\
    \vp'
    \end{pmatrix}\right\|_2^3 }\,. 
$$
Suppose further that there exists positive constants $\lambda, M_2 \in \RR_+$ such that 
\begin{align*}
    \mH(\vx) &\succeq \lambda \sI\,,\\
    \mN(\vx)^T \nabla^2 \gI(\vx,\vp)
    \mN(\vx) &\preceq M_2\sI
\end{align*}
for all $\vx\in \RR^n$. If $\tau = a \frac{\lambda}{M_2+\delta}$ for some $a \in (0,2)$, then the functional $\gI(\vx^n,\vp^n)$ decreases geometrically, i.e. 
$$\gI(\vx^{n},\vp^{n}) \leq \gI(\vx^{0},\vp^{0})\big(1+\frac{a^2-2a}{2}\frac{\lambda^2}{M_2+\delta}\big)^n\,.
$$  
\end{theorem}
\begin{remark}
Note that the constant $M_2$ in Theorem \ref{thm:discrete_exp_decay_2} can be better than the constant $M_1$ in Theorem \ref{thm:discrete_exp_decay_1} because $\mN$ and $\nabla^2 \gI$ are evaluated at the same $\vx$ in Theorem \ref{thm:discrete_exp_decay_2}. 
\end{remark}
\begin{proof}
We will prove it by induction. Using the mean-value theorem, we have 
\begin{align}\label{eq:I_third_taylor_upperbound}
    &\gI(\vx^{n+1},\vp^{n+1})-\gI(\vx^{n},\vp^{n}) \nonumber \\
    = &\begin{pmatrix}
    \nabla_{\vx}\gI(\vx^n,\vp^n)\\
    \nabla_{\vp}\gI(\vx^n,\vp^n)
    \end{pmatrix}^T 
    \begin{pmatrix}
    \vx^{n+1}-\vx^n \\
    \vp^{n+1} - \vp^n
    \end{pmatrix} + \frac{1}{2}\begin{pmatrix}
    \vx^{n+1}-\vx^n \\
    \vp^{n+1} - \vp^n
    \end{pmatrix}^T  \nabla^2 \gI(\vx^n,\vp^n)\begin{pmatrix}
    \vx^{n+1}-\vx^n \\
    \vp^{n+1} - \vp^n 
    \end{pmatrix} \nonumber \\
    &\qquad + \frac{1}{6} \begin{pmatrix}
    \vx^{n+1}-\vx^n \\
    \vp^{n+1} - \vp^n
    \end{pmatrix}^T \left(\nabla^3 \gI(\tilde{\vx}^n,\tilde{\vp}^n)    \begin{pmatrix}
    \vx^{n+1}-\vx^n \\
    \vp^{n+1} - \vp^n
    \end{pmatrix} \right) \begin{pmatrix}
    \vx^{n+1}-\vx^n \\
    \vp^{n+1} - \vp^n
    \end{pmatrix} \,,
\end{align}
where $(\tilde{\vx}^n,\tilde{\vp}^n)$ is in between $(\vx^{n+1},\vp^{n+1})$ and $(\vx^n,\vp^n)$. By \eqref{eq:I_third_taylor_upperbound} and \eqref{eq:xp_N_gradf_p}, we can bound  
\begin{align}\label{eq:I_third_taylor_inexact}
    &\gI(\vx^{1},\vp^{1})-\gI(\vx^{0},\vp^{0})\nonumber \\
    =&-\tau \begin{pmatrix}
    \nabla f(\vx^0) \\
    \vp^0
    \end{pmatrix}^T \mH(\vx^0) \begin{pmatrix}
    \nabla f(\vx^0) \\
    \vp^0
    \end{pmatrix}\nonumber \\
    &+\frac{\tau^2}{2}\begin{pmatrix}
    \nabla f(\vx^0) \\
    \vp^0
    \end{pmatrix}^T
    \mN(\vx^0)^T \nabla^2 \gI(\vx^0,\vp^0)
    \mN(\vx^0)
    \begin{pmatrix}
    \nabla f(\vx^0) \\
    \vp^0
    \end{pmatrix}\nonumber \\
    &- \frac{\tau^3}{6} \begin{pmatrix}
    \nabla f(\vx^0) \\
    \vp^0
    \end{pmatrix}^T
    \mN(\vx^0)^T \left( \nabla^3 \gI(\tilde{\vx}^0,\tilde{\vp}^0)
    \mN(\vx^0)
    \begin{pmatrix}
    \nabla f(\vx^0) \\
    \vp^0
    \end{pmatrix}\right)\mN(\vx^0)
    \begin{pmatrix}
    \nabla f(\vx^0) \\
    \vp^0
    \end{pmatrix} \nonumber \\
    \leq & -\tau \begin{pmatrix}
    \nabla f(\vx^0) \\
    \vp^0
    \end{pmatrix}^T\mH(\vx^0) \begin{pmatrix}
    \nabla f(\vx^0) \\
    \vp^0
    \end{pmatrix}\nonumber \\
    &+\frac{\tau^2}{2}\begin{pmatrix}
    \nabla f(\vx^0) \\
    \vp^0
    \end{pmatrix}^T
    \mN(\vx^0)^T \nabla^2 \gI(\vx^0,\vp^0)
    \mN(\vx^0)
    \begin{pmatrix}
    \nabla f(\vx^0) \\
    \vp^0
    \end{pmatrix}\nonumber \\
    &+ \frac{\tau^3}{6} \begin{pmatrix}
    \nabla f(\vx^0) \\
    \vp^0
    \end{pmatrix}^T
    \left( D_0 
    \|\mN(\vx^0)\|_2^3 \left\|
    \begin{pmatrix}
    \nabla f(\vx^0) \\
    \vp^0
    \end{pmatrix} \right\|_2\right)
    \begin{pmatrix}
    \nabla f(\vx^0) \\
    \vp^0
    \end{pmatrix}\nonumber \\
    =& -\tau \begin{pmatrix}
    \nabla f(\vx^0) \\
    \vp^0
    \end{pmatrix}^T \mH(\vx^0) \begin{pmatrix}
    \nabla f(\vx^0) \\
    \vp^0
    \end{pmatrix}\nonumber \\
    &+\frac{\tau^2}{2}\begin{pmatrix}
    \nabla f(\vx^0) \\
    \vp^0
    \end{pmatrix}^T
    \mN(\vx^0)^T \nabla^2 \gI(\vx^0,\vp^0)
    \mN(\vx^0)
    \begin{pmatrix}
    \nabla f(\vx^0) \\
    \vp^0
    \end{pmatrix}\nonumber \\
    &+ \frac{\tau^3}{6} \begin{pmatrix}
    \nabla f(\vx^0) \\
    \vp^0
    \end{pmatrix}^T
      D_0 
    \|\mN(\vx^0)\|_2^3 \gI(\vx_0,\vp_0)^{1/2} 
    \begin{pmatrix}
    \nabla f(\vx^0) \\
    \vp^0
    \end{pmatrix} \nonumber \\
    \leq & -\tau \begin{pmatrix}
    \nabla f(\vx^0) \\
    \vp^0
    \end{pmatrix}^T \mH(\vx^0) \begin{pmatrix}
    \nabla f(\vx^0) \\
    \vp^0
    \end{pmatrix}\nonumber \\
    &+\frac{\tau^2}{2}\begin{pmatrix}
    \nabla f(\vx^0) \\
    \vp^0
    \end{pmatrix}^T
    \mN(\vx^0)^T \nabla^2 \gI(\vx^0,\vp^0)
    \mN(\vx^0)
    \begin{pmatrix}
    \nabla f(\vx^0) \\
    \vp^0
    \end{pmatrix}\nonumber \\
    &+ \frac{\tau^2 \delta}{6} \begin{pmatrix}
    \nabla f(\vx^0) \\
    \vp^0
    \end{pmatrix}^T
    \begin{pmatrix}
    \nabla f(\vx^0) \\
    \vp^0
    \end{pmatrix}\,,
\end{align}
where the last inequality is by our assumption on $(\vx^0,\vp^0)$. 
Using our assumptions on the lower bound of $\mH$ and the upper bound of $\mN^T \cdot \nabla^2 \gI \cdot \mN$, we obtain 
\begin{align}\label{eq:I0_I1}
    \gI(\vx^{1},\vp^{1})-\gI(\vx^{0},\vp^{0}) &\leq \big(-\tau \lambda + \frac{\tau^2 \delta}{6} + \frac{\tau^2 M_2 }{2} \big) \gI(\vx^{0},\vp^{0}) \nonumber \\
    &\leq \big(-\tau \lambda + \frac{\tau^2 (\delta+M_2)}{2}\big) \gI(\vx^{0},\vp^{0}) \nonumber \\
    &=\frac{1}{2}(a^2-2a)\frac{\lambda^2}{M_2+\delta}\gI(\vx^{0},\vp^{0})\,,
\end{align}
where we used $\tau = a \frac{\lambda}{M_2+\delta}$ for some $a\in (0,2)$. Hence,
$$\gI(\vx^{1},\vp^{1}) \leq \gI(\vx^{0},\vp^{0})\big(1+\frac{a^2-2a}{2}\frac{\lambda^2}{M_2+\delta}\big)\,.
$$
This proves the base case. Now suppose it holds that  
$$\gI(\vx^{n},\vp^{n}) \leq \gI(\vx^{0},\vp^{0})\big(1+\frac{a^2-2a}{2}\frac{\lambda^2}{M_2+\delta}\big)^n\,,
$$
for some $n\geq 1$. In particular, this implies that 
$$\gI(\vx^{n},\vp^{n}) < \gI(\vx^{0},\vp^{0})\,,
$$ 
which yields 
$$\tau D_0 \|\mN(\vx)\|_2^3 \gI(\vx^n,\vp^n)^{1/2} < \tau D_0 \|\mN(\vx)\|_2^3 \gI(\vx^0,\vp^0)^{1/2} \leq \delta \,.
$$
Then, repeating the derivation of \eqref{eq:I_third_taylor_inexact} and \eqref{eq:I0_I1} yields
$$\gI(\vx^{n+1},\vp^{n+1}) \leq \gI(\vx^{n},\vp^{n})\big(1+\frac{a^2-2a}{2}\frac{\lambda^2}{M_2+\delta}\big)\,.
$$ 
Combining with our induction hypothesis, we conclude that 
$$ \gI(\vx^{n+1},\vp^{n+1}) \leq \gI(\vx^{0},\vp^{0})\big(1+\frac{a^2-2a}{2}\frac{\lambda^2}{M_2+\delta}\big)^{n+1}\,.
$$ 
The proof is complete by induction.
\end{proof}
\begin{corollary}\label{cor:H}
Suppose Assumption \ref{assumption:mu_L} and Assumption \ref{assumption:L'} hold. When $\sigma=\tau$, $\gamma = \frac{1-\sigma \mu }{\mu}, \varepsilon = 1, A =  \frac{\mu+L}{2+(\mu+L)\varepsilon\gamma}$, we have 
$$\mH(\vx) \succeq \frac{\mu}{4}\sI.
$$
\end{corollary}
\begin{proof}
 By definition of $\mH$, we can compute 
$$ (1+\sigma \varepsilon A)\cdot\mH(\vx) = \begin{pmatrix}
    \mC_0(\vx)(\sigma\mA+\gamma \mA) & \frac{1}{2}\mC_0(\vx)(1-\varepsilon\gamma A)-\frac{1}{2} \eta\mA\\
    \frac{1}{2}\mC_0(\vx)(1-\varepsilon\gamma A)-\frac{1}{2} \eta\mA &\eta\varepsilon \mA
\end{pmatrix}\,,
$$
where $\eta=\sigma/\tau=1$, $\mC_0(\vx) = \nabla^2 f(\vx) \mB(\vx) \nabla^2f(\vx)$. We want to find some constant $\lambda>0$, such that 
$$
 \begin{pmatrix}
     \vz \\ \vw 
 \end{pmatrix}^T \mH(\vx) \begin{pmatrix}
     \vz \\ \vw 
 \end{pmatrix} \geq \lambda (\|\vz\|^2+\|\vw\|^2)\,.
$$
Observe that 
\begin{align}
    &\begin{pmatrix}
     \vz \\ \vw 
 \end{pmatrix}^T \mH(\vx) \begin{pmatrix}
     \vz \\ \vw 
 \end{pmatrix} -\lambda (\|\vz\|^2+\|\vw\|^2) \nonumber \\
 &= \vz^T \big(\mC_0 (\gamma+\sigma) A/(1+\sigma \varepsilon A) -\lambda \sI \big)\vz  + \vw^T\big(\varepsilon A /(1+\sigma \varepsilon A)  - \lambda \sI \big) \vw \nonumber \\
    &\qquad + \vz^T \big(-A + \mC_0(\sI-\varepsilon \gamma A ) \big)\vw/(1+\sigma \varepsilon A) \,,
\end{align}
which is almost the same as \eqref{eq:Lyapunov_continuous_derivative}. Thus, following a similar procedure in Theorem \ref{thm:continuous_Lyapunov_eps} with the provided parameters, we obtain that  
$$
\lambda \geq \frac{\mu}{2}\frac{1+\frac{A\sigma}{2}}{1+\sigma A} \geq \frac{\mu}{4}\,.
$$
This implies 
\begin{align}
    \mH(\vx) \succeq \frac{\mu}{4}\sI\,.
\end{align}
\end{proof}
\begin{corollary}\label{cor:NIN}
Let $f:\RR^d \to \RR$ be a $\mathcal{C}^4$ strongly convex function. Suppose Assumption \ref{assumption:mu_L} and Assumption \ref{assumption:L'} hold. If $\sigma=\tau$, $\gamma = \frac{1-\sigma \mu }{\mu}, \varepsilon = 1, A =  \frac{\mu+L}{2+(\mu+L)\varepsilon\gamma}$, we have  
\begin{enumerate}
\item[(1)] 
$$
\| \mN(\vx) \|_2 \leq \frac{\max\{L,1\} \cdot \big(A(\sigma + 2\gamma + 2) + 1\big)}{(1+\sigma A)}\,.
$$
    \item[(2)] $$
\mN(\vx)^T \nabla^2 \gI(\vx,\vp)
    \mN(\vx) \preceq  \frac{(3+\sigma A+2A)^2}{(1+\sigma  A)^2}\cdot  \max\{L',1\}\cdot  \sI\, . 
$$
\end{enumerate}

\end{corollary}
\begin{proof}

We can decompose 
$$(1+\sigma  A)\cdot \mN(\vx) = \begin{pmatrix}
    \mB(\vx) \nabla^2 f(\vx) & 0 \\
    0 & \sI 
\end{pmatrix}
\begin{pmatrix}
    (\sigma+\gamma)\mA & (\sI-\gamma\mA) \\
    - \mA &   \mA
\end{pmatrix}\,.
$$
Observe that 
$$
\begin{pmatrix}
    (\sigma+\gamma)\mA & (\sI-\gamma\mA) \\
    - \mA &   \mA
\end{pmatrix} = \begin{pmatrix}
    (\sigma+\gamma)\sI & -\gamma \sI \\
    -\sI & \sI 
\end{pmatrix} \cdot \begin{pmatrix}
    \mA & 0 \\
    0 &\mA 
\end{pmatrix} + \begin{pmatrix}
    0&\sI\\
    0&0
\end{pmatrix}\,.
$$
Therefore, \begin{align} \label{eq:N_norm}
\bigg\| \begin{pmatrix}
    (\sigma+\gamma)\mA & (\sI-\gamma\mA) \\
    - \mA &   \mA
\end{pmatrix} \bigg\|_2 & \leq  A \bigg\|\begin{pmatrix}
    (\sigma+\gamma)\sI & -\gamma \sI \\
    -\sI & \sI 
\end{pmatrix} \bigg\|_2 + \bigg\|\begin{pmatrix}
    0&\sI\\
    0&0
\end{pmatrix} \bigg\|_2 \nonumber \\
&\leq A \bigg\|\begin{pmatrix}
    (\sigma+\gamma)\sI & -\gamma \sI \\
    -\sI & \sI 
\end{pmatrix} \bigg\|_2 + 1 \nonumber \\
&\leq A(\sigma + 2\gamma + 2) + 1\,.
\end{align}
To get the last inequality, we consider $(\vz,\vw)$ such that $\|\vz\|^2 + \|\vw\|^2 = 1$. Thus 
\begin{align*}
    \bigg\| \begin{pmatrix}
    (\sigma+\gamma)\sI & -\gamma \sI \\
    -\sI & \sI 
\end{pmatrix} \begin{pmatrix}
    \vz \\ \vw
\end{pmatrix} \bigg\|_2 &= \bigg\| \begin{pmatrix}
    (\sigma + \gamma)\vz - \gamma \vw \\ -\vz+\vw 
\end{pmatrix} \bigg\| \\
&\leq \sigma \bigg\| \begin{pmatrix}
    \vz \\ \vw 
\end{pmatrix}\bigg\| + \gamma \| \vz-\vw \| + \| \vz - \vw\| \\ 
&\leq \sigma + 2\gamma + 2 \,. 
\end{align*}
We now have 
\begin{align*}
    (1+\sigma  A) \| \mN(\vx)\|_2  &\leq  \bigg\|\begin{pmatrix}
    \mB(\vx) \nabla^2 f(\vx) & 0 \\
    0 & \sI 
\end{pmatrix}\bigg\|_2
\bigg\| \begin{pmatrix}
    (\sigma+\gamma)\mA & (\sI-\gamma\mA) \\
    - \mA &   \mA
\end{pmatrix} \bigg\|_2 \\ 
&\leq \max\{L,1\} \cdot \big(A(\sigma + 2\gamma + 2) + 1\big)\,.
\end{align*}
This proves part (1) of our Corollary. It follows that 
\begin{align}\label{eq:N_HessianI_N}
    &\mN(\vx)^T \nabla^2 \gI(\vx,\vp)
    \mN(\vx)\nonumber \\
    &= \frac{1}{(1+\sigma A)^2} \begin{pmatrix}
    (\sigma+\gamma)\mA &  - \mA\\ 
    (\sI-\gamma\mA) &   \mA
\end{pmatrix} \cdot \begin{pmatrix}
    \mC(\vx)^T\big(\nabla^3 f(\vx) \nabla f(\vx)+(\nabla^2 f(\vx))^2\big)\mC(\vx) & 0 \\
    0& \sI
    \end{pmatrix}\nonumber \\
    &\hspace{2cm}  \cdot\begin{pmatrix}
    (\sigma+\gamma)\mA & (\sI-\gamma\mA) \\
    - \mA &   \mA
\end{pmatrix}\,,
\end{align}
where we recall $\mC(\vx) =  \mB(\vx) \nabla^2 f(\vx)$. 

By assumption, there exists $L'>0$, such that 
$$
\mC(\vx)^T\big(\nabla^3 f(\vx) \nabla f(\vx)+(\nabla^2 f(\vx))^2\big)\mC(\vx) \preceq L' \sI\, ,
$$
for all $\vx$. Then, combining \eqref{eq:N_HessianI_N} and \eqref{eq:N_norm}, we obtain that 
\begin{align}
    \|\mN(\vx)^T \nabla^2 \gI(\vx,\vp)
    \mN(\vx)\|_2 &\leq \frac{\big(A(\sigma+2\gamma + 2)+1\big)^2}{(1+\sigma  A)^2}\cdot  \max\{L',1\} \nonumber \\ 
    &\leq \frac{(3+\sigma A+2A)^2}{(1+\sigma  A)^2}\cdot  \max\{L',1\}\,,
\end{align}
where we have used $\gamma A <1$ to derive the last inequality. 
\end{proof}

\begin{theorem}[Restatement of Theorem \ref{thm:informal}]\label{thm:restatement}
Suppose Assumption \ref{assumption:mu_L} and Assumption \ref{assumption:L'} hold. Let $\sigma=\tau$, $\gamma = \frac{1-\sigma \mu }{\mu}, \varepsilon = 1, A =  \frac{\mu+L}{2+(\mu+L)\varepsilon\gamma}$. And suppose \peqref{condition:x0p0} holds for some $\delta>0$ and all $\vx$. If $\tau = \frac{1}{4}\frac{\mu}{\delta + 36\max\{L',1\}}$, then 
$$
\gI(\vx^n,\vp^n) \leq \gI(\vx^0,\vp^0) \big(1 -\frac{\mu^2/32}{\delta + 36 \max\{L',1\}}\big)^n\, .
$$
\end{theorem}
\begin{proof}
By Assumption \ref{assumption:mu_L} and Assumption \ref{assumption:L'}, we have $\mu \leq L \leq L'$. Thus $\mu/L'\leq 1$ and $\sigma = \tau <1/36$. Moreover, 
$$
\gamma = \frac{1}{\mu} - \sigma \geq 1-\frac{1}{36}=\frac{35}{36}\,.
$$
 And 
$$
A = \frac{\mu+L}{2+(\mu+L)\varepsilon\gamma} < \frac{1}{\gamma} \leq \frac{36}{35}\,.
$$
Then it follows  
$$
\frac{3+\sigma A+2A}{1+\sigma A} \leq 3+\sigma A+2A < 6\,.
$$
By Corollary \ref{cor:NIN}, we have 
$$
\|\mN(\vx)^T \nabla^2 \gI(\vx,\vp)
    \mN(\vx)\|_2 \leq 36 \max\{L',1\}\,.
$$
Combining this with  Theorem \ref{thm:discrete_exp_decay_2} and Corollary \ref{cor:H}, we finish the proof.
    
\end{proof}
\begin{remark}
    The choice of parameters in Theorem \ref{thm:restatement} may not be optimal. The main purpose of Theorem \ref{thm:restatement} is to show the existence of geometric convergence in Algorithm \ref{alg:pdd}. 
\end{remark}

\section{Numerical experiment}\label{section:numerics}
We test our PDD algorithm using several convex and non-convex functions and compare the results with other commonly used optimizers, such as gradient descent, Nesterov's accelerated gradient (NAG), IGAHD (inertial gradient algorithm with Hessian damping) \citep{attouch2020first}, and IGAHD-SC (inertial gradient algorithm with Hessian damping for strongly convex functions) \citep{attouch2020first}. 

\subsection{Summary of algorithms}
For reference, we write down the iterations of gradient descent, NAG, IGAHD-SC, and IGAHD for better comparison. 

\textbf{Gradient descent}: 
\begin{equation*}
    \vx^{n+1} = \vx^n - \tau_{\textrm{gd}} \nabla f(\vx^n)\, , 
\end{equation*}
where $\tau_{\textrm{gd}}>0$ is a stepsize. 

\textbf{NAG}: 
\begin{align}
    \vy^{n+1} &= \vx^n - \tau_{\textrm{nag}} \nabla f(\vx^n) \,,\nonumber \\ 
    \vx^{n+1} &= \vy^{n+1} + \beta_{\textrm{nag}}(\vy^n - \vy^{n-1})\, , \nonumber 
\end{align} 
where $\tau_{\textrm{nag}}>0$ is a stepsize, and $\beta_{\textrm{nag}}>0$ is a parameter. 

\textbf{IGAHD}: Suppose $\nabla f$ is $L_1$-Lipschitz. 
\begin{align*}
    \vy^n &= \vx^n + \alpha_n(\vx^n - \vx^{n-1}) - \beta^{(1)} \sqrt{\tau_{\textrm{att}}}(\nabla f(\vx^n) - \nabla f(\vx^{n-1})) - \frac{\beta^{(1)} \sqrt{\tau_{\textrm{att}}}}{n}\nabla f(\vx^{n-1})\,, \\
    \vx^{n+1} &= \vy^n - \tau_{\textrm{att}}\nabla f(\vy^n) \,.
\end{align*}
Here $\alpha_n = 1-\frac{\alpha}{n}$ for some $\alpha\geq 3$. $\beta^{(1)}$ needs to satisfy 
$$
0\leq \beta^{(1)} \leq 2\sqrt{\tau_{\textrm{att}}}\,.
$$
And $\tau_{\textrm{att}}>0$ is a stepsize, which needs to satisfy 
$$
\tau_{\textrm{att}} \leq \frac{1}{L_1}\,. 
$$
\begin{remark}
    As mentioned earlier, in each iteration of IGAHD, $\nabla f(\cdot)$ is evaluated twice: at $\vx^n$ and at $\vy^n$. 
    This differs from one gradient evaluation in gradient descent, NAG, and our method \eqref{eq:discrete_linearized_pdhg}. \citet{chen2021unified} proposed a slightly different algorithm from IGAHD that only requires one gradient evaluation  in each iteration. %
\end{remark}
\textbf{IGHD-SC}: Suppose $f$ is $m_1$-strongly convex and $\nabla f$ is $L_1$-Lipschitz. 
\begin{equation*}
    \vx^{n+1} = \vx^n + \frac{1-\sqrt{m_1 \tau_{\textrm{att}}}}{1+\sqrt{m_1 \tau_{\textrm{att}}}}(\vx^n-\vx^{n-1}) - \frac{\beta^{(2)}\sqrt{\tau_{\textrm{att}}}}{1+\sqrt{m_1 \tau_{\textrm{att}}}}(\nabla f(\vx^n) - \nabla f(\vx^{n-1})) - \frac{\tau_{\textrm{att}}}{1+\sqrt{m_1 \tau_{\textrm{att}}}}\nabla f(\vx^n)\,.
\end{equation*}
Here $\beta^{(2)}$ and $L_1$ need to satisfy 
\begin{align}
    \beta^{(2)} \leq \frac{1}{\sqrt{m_1}}\, ,\quad 
    L_1 \leq \min \Big\{\frac{\sqrt{m_1}}{8\beta^{(2)}},\frac{\frac{\sqrt{m_1}}{2\tau_{\textrm{att}}}+\frac{m_1}{\sqrt{\tau_{\textrm{att}}}}}{2\beta^{(2)} m_1 + \frac{1}{\sqrt{\tau_{\textrm{att}}}}+ \frac{\sqrt{m_1}}{2}}   \Big\}  \,.
\end{align}

\subsection{regularized log-sum-exp}
Consider the regularized log-sum-exp function $$
f(\vx) = \log\left( \sum_{i=1}^n \exp(\vq_i ^T \vx) \right) + \frac{1}{2}\vx^T \mQ \vx \,,
$$
where $n=100$, $\mQ = \mQ^T \succ 0$ and $\vq_i^T$ is the ith row of $\mQ$. $\mQ$ is chosen to be diagonally dominant, i.e. $Q_{i,i} > \sum_{j\neq i} |Q_{i,j}|$. In this case, we may choose the diagonal preconditioner $\mC(\vx) = \big(\textrm{diag} (\mQ)\big)^{-1}$. We compare the performance of gradient descent, preconditioned gradient descent, PDD with $\mC(\vx) = \sI$, PDD with diagonal preconditioner, NAG, and IGAHD-SC (inertial gradient algorithm with Hessian damping for strongly convex functions) by \citet{attouch2020first} methods for minimizing $f$. The stepsize of gradient descent is determined by $\tau_{\textrm{gd}} = \frac{2}{\lambda_1*3 + \lambda_n}$, where $\lambda_1$ and $\lambda_n$ are the maximum and minimum eigenvalues of $\mQ$, respectively. For a pure quadratic objective function, $\vx^T \mQ \vx$, the optimal stepsize of gradient descent is $\frac{2}{\lambda_1 + \lambda_n}$. However, since our objective function also contains a log-sum-exp term, we slightly change the stepsize. Otherwise, gradient descent will not converge. Similarly, when deciding the parameters for NAG, we choose $\tau_{\textrm{nag}} = \frac{4}{30*\lambda_1 + \lambda_n}$ and $\beta_{\textrm{nag}} = \frac{\sqrt{3\kappa'+1}-2}{\sqrt{3\kappa'+1}+2}$, where $\kappa' = 10\lambda_1 / \lambda_n$, which is slightly smaller than the optimal parameters of NAG for a purely quadratic function to guarantee convergence. For PDD with $\mC(\vx) = \sI$, we choose $\tau_{\textrm{pdd}}=\sigma_{\textrm{pdd}} = \frac{2}{\lambda_1 + \lambda_n}$, $\varepsilon = 1$, $A = 10$, $\omega = 1$. For PDD with diagonal preconditioner $\mC(\vx) = \big(\textrm{diag} (\mQ)\big)^{-1}$, we choose $\tau_{\textrm{pdd}}=\sigma_{\textrm{pdd}} = 0.5$, $\varepsilon = 1$, $A = 1$, $\omega = 1$. We use the same $\mC(\vx) = \big(\textrm{diag} (\mQ)\big)^{-1}$ as a preconditioner for gradient descent. The stepsize for preconditioned gradient descent is chosen to be the same as $\tau_{\textrm{pdd}} = 0.5$. For IGAHD-SC (`att'), we need $m_1$ as the smallest eigenvalue of $\nabla^2 f(\vx)$. In this example, we may estimate $m_1$ as the smallest eigenvalue of $\mQ$. And $\tau_{\textrm{att}}= 0.0016$ via grid search. $\beta^{(2)}$ in IGAHD-SC is found by solving (see Theorem 11 Eq.~(26) of \citet{attouch2020first})
$$
\frac{\sqrt{m_1}}{8\beta^{(2)}} = \frac{\frac{\sqrt{m_1}}{2\tau_{\textrm{att}}}+\frac{\sqrt{m_1}}{\sqrt{\tau_{\textrm{att}}}}}{2\beta^{(2)} m_1 + \frac{1}{\sqrt{\tau_{\textrm{att}}}}+\frac{\sqrt{m_1}}{2}}\,,
$$
which gives 
\begin{equation}\label{eq:beta_equation}
    \beta^{(2)} = \frac{\sqrt{\tau_{\textrm{att}}}+\tau_{\textrm{att}}\sqrt{m_1}/2}{4+8\sqrt{m_1}\sqrt{\tau_{\textrm{att}}}-2m_1 \tau_{\textrm{att}}}\,.
\end{equation}

The initial condition is $\vx^0 = \textrm{np.ones}(n)*0.1$. The result is presented in Fig.~\ref{fig:convex1}. 

\subsection{Quadratic minus cosine function}
Consider the function 
$$ f(\vx) = \|\vx\|^2 - \cos(\vc^T \vx)\,,
$$
where $\vc$ is a vector in $\RR^{100}$ with $\|\vc\|^2 = 1.9$. Then a direct calculation shows that $ 0.1\sI \preceq \nabla^2 f(\vx) \preceq 3.9 \sI$ for any $\vx$. This allows us to choose the optimal stepsize for gradient descent and NAG. When minimizing $f$ using gradient descent, we can choose $\tau_{\textrm{gd}} = \frac{2}{0.1+3.9} = 0.5$. Meanwhile, for NAG, we may choose $\tau_{\textrm{nag}} = \frac{4}{3*3.9 + 0.1}$, and $\beta = \frac{\sqrt{3\kappa+1}-2}{\sqrt{3\kappa+1}+2}$, where $\kappa = 3.9 / 0.1$. For PDD with $\mC(\vx) = \sI$, we choose $\tau_{\textrm{pdd}}=\sigma_{\textrm{pdd}} = 0.5$, $\varepsilon = 1$, $A = 1$, $\omega = 1$. For IGAHD-SC (`att'), we choose $m_1=0.1$, $\tau_{\textrm{att}} = 0.55$ via grid search and $\beta^{(2)}$ is given by \eqref{eq:beta_equation}. The initial condition is $\vx^0 = \textrm{np.ones}(n)*5$. The result is presented in Fig.~\ref{fig:convex2}. 
\begin{figure}
\centering
\begin{subfigure}{.5\textwidth}
  \centering
  \includegraphics[width=\linewidth]{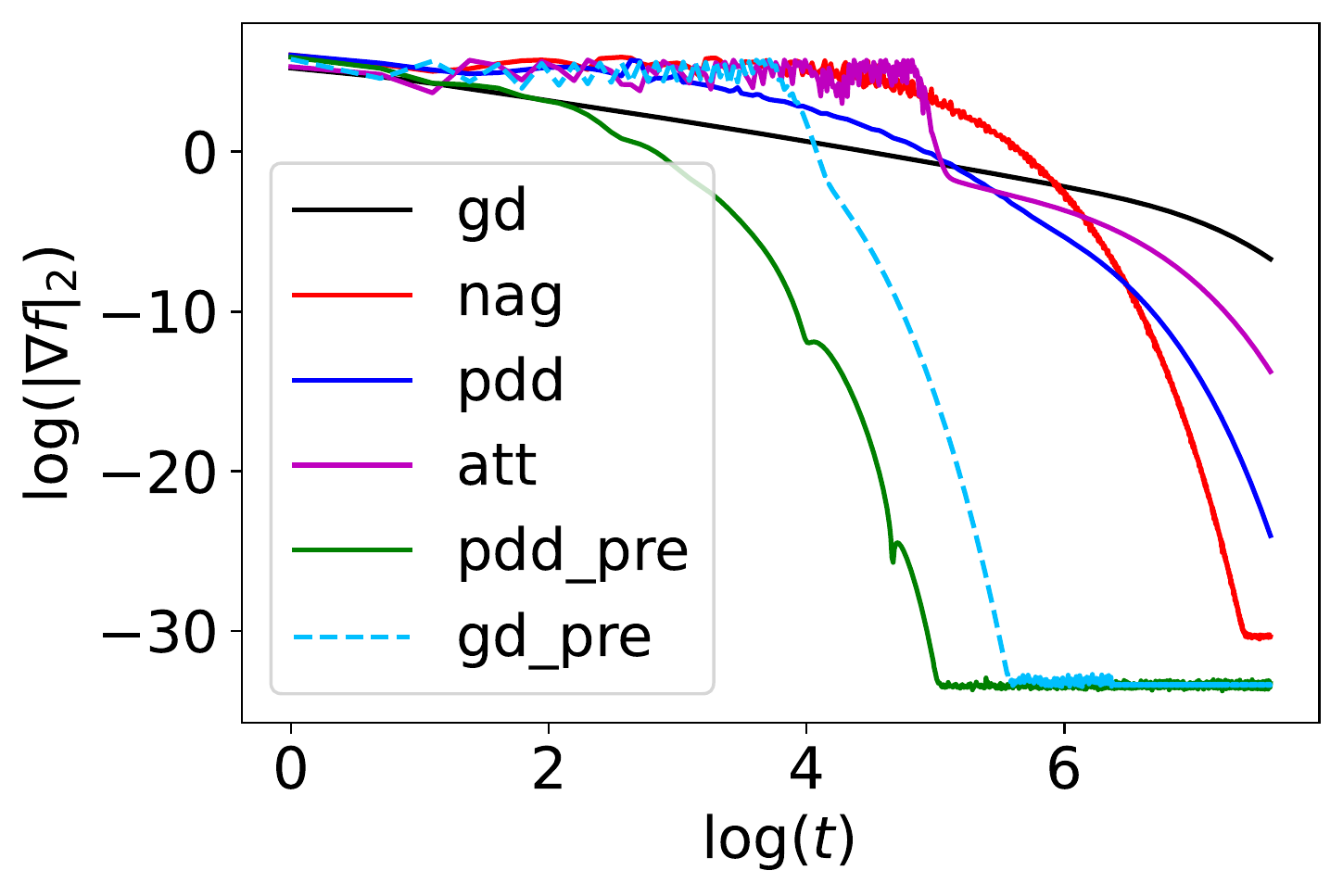}
  \caption{Regularized log-sum-exp}
  \label{fig:convex1}
\end{subfigure}%
\begin{subfigure}{.5\textwidth}
  \centering
  \includegraphics[width=\linewidth]{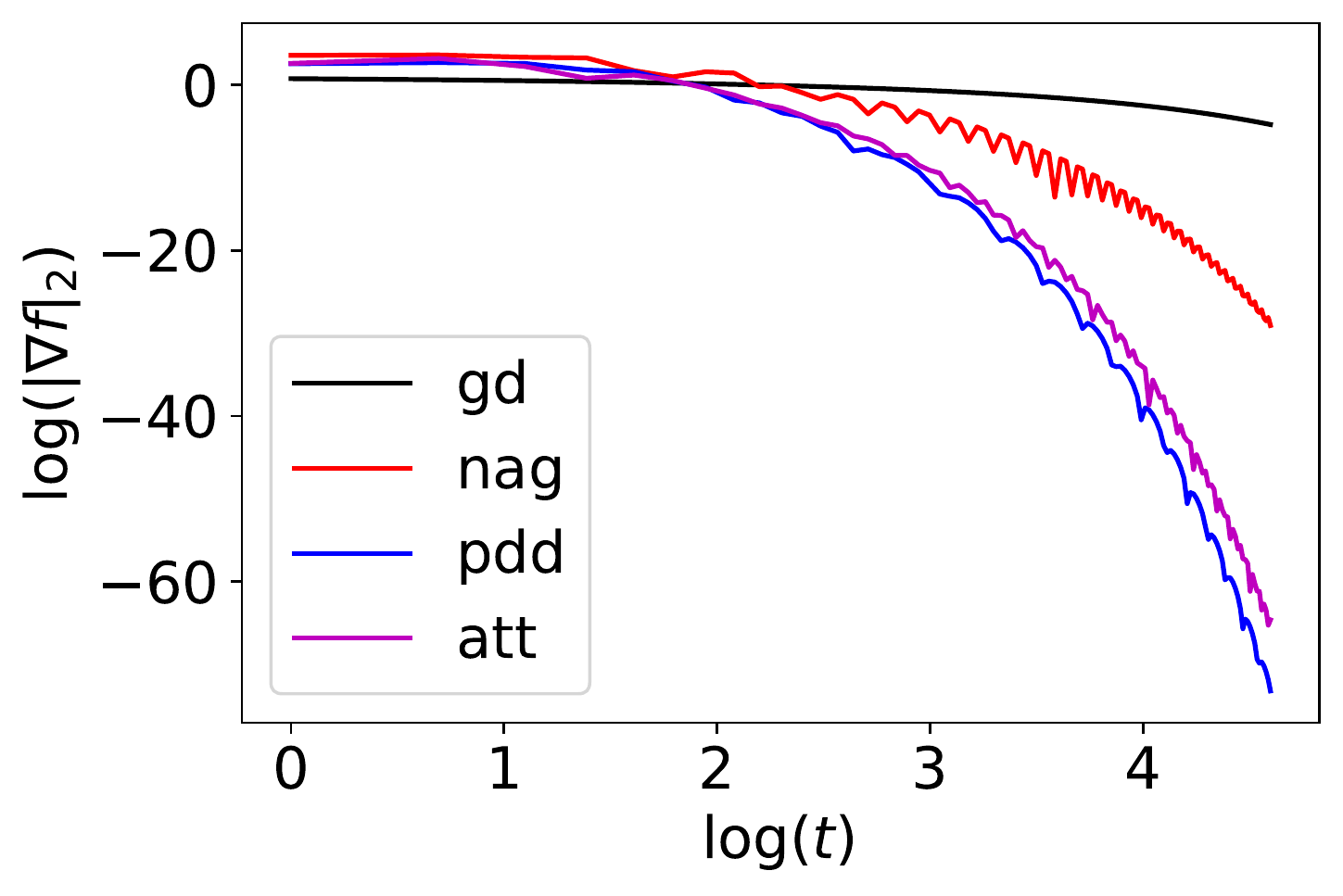}
  \caption{Qudratic minus cosine}
  \label{fig:convex2}
\end{subfigure}
\caption{Comparison of gradient descent, NAG, PDD, and IGAHD-SC (we use `att' as a shorthand for this method) on minimizing (a) the regularized log-sum-exp function and (b) the quadratic minus cosine function. The y-axis represents the 2-norm of the gradient of the objective function on a logarithmic scale. The x-axis represents the number of iterations on a logarithmic scale.}
\label{fig:convex}
\end{figure}
\subsection{Rosenbrock function}
\subsubsection{2-dimension}
The 2-dimensional Rosenbrock function is defined as 
$$ f(x,y) = (a-x)^2 + b(y-x^2)^2\,,
$$
where a common choice of parameters is $a = 1$, $b = 100$. This is a non-convex function with a global minimum of $(x^*,y^*) = (a,a^2)$. The global minimum is inside a long, narrow, parabolic-shaped flat valley. To find the valley is trivial. To converge to the global minimum, however, is difficult. We compare the performance of gradient descent, NAG, PDD with $\mC(\vx) = \sI$ and IGAHD (inertia gradient algorithm with Hessian damping) by \citet{attouch2020first} when minimizing the Rosenbrock function starting from $(-3,-4)$. The stepsize of gradient descent is $\tau_{\textrm{gd}} = 0.0002$. The stepsize of NAG is $\tau_{\textrm{nag}} = 0.0002$, $\beta_{\textrm{nag}} = 0.9$. The parameters of PDD are $\tau_{\textrm{pdd}} = \sigma_{\textrm{pdd}}= 0.005$, $\varepsilon = 1$, $\omega = 1$, $A = 5$. The stepsize of the PDD method is larger than $\tau_{\textrm{gd}}$ and $\tau_{\textrm{nag}}$ because gradient descent and NAG do not allow larger stepsizes for convergence. For IGAHD (`att'), we choose $\tau_{\textrm{att}} = 0.00045$, $\alpha = 3$, $\beta^{(1)} = \sqrt{\tau_{\textrm{att}}}/14$. The convergence result and the optimization trajectories are shown in Fig.~\ref{fig:rosenbrock}. 
\begin{figure}
\centering
\begin{subfigure}{.5\textwidth}
  \centering
  \includegraphics[width=\linewidth]{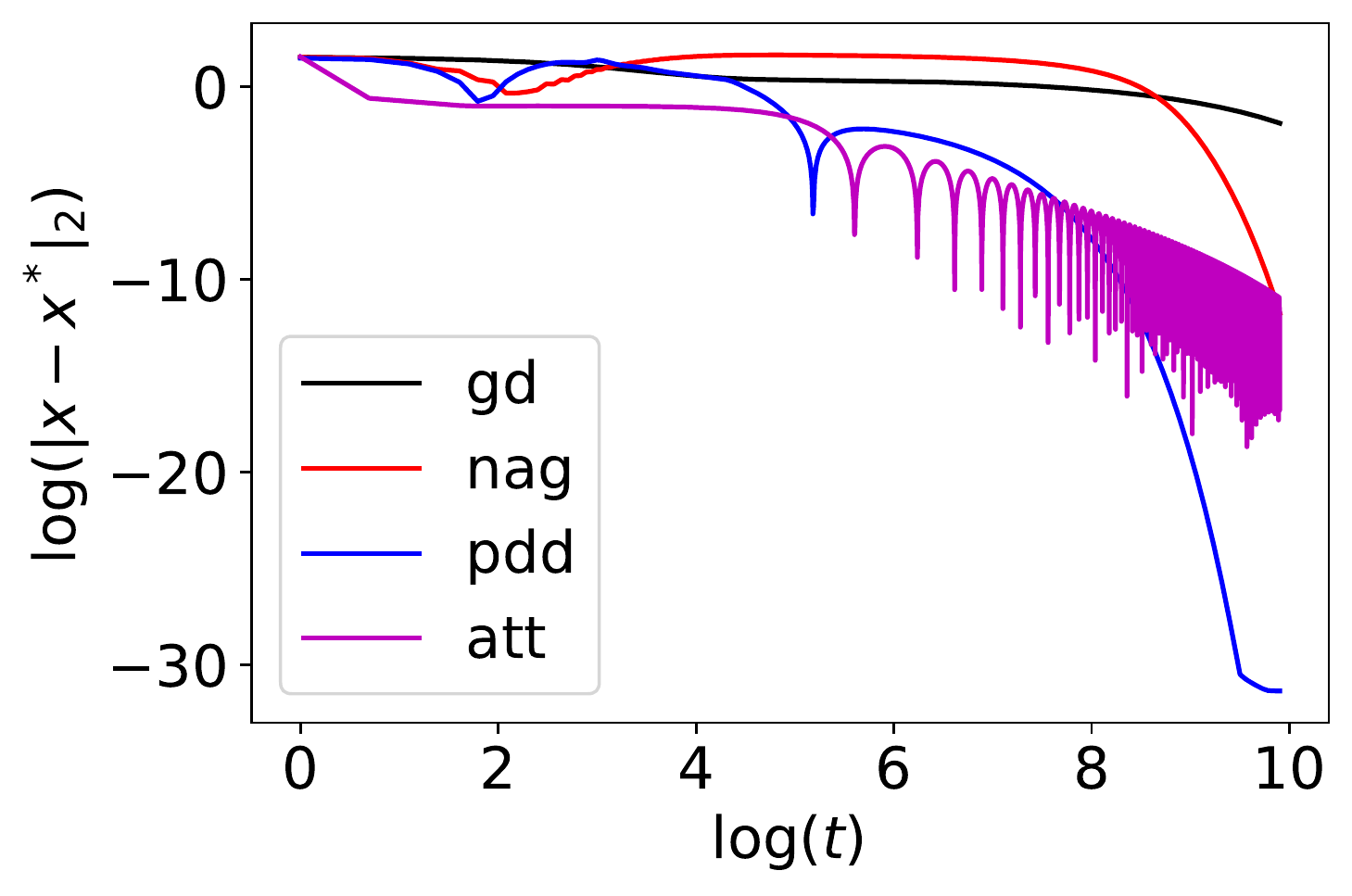}
  \caption{Convergence comparison}
  \label{fig:rosen1}
\end{subfigure}%
\begin{subfigure}{.5\textwidth}
  \centering
  \includegraphics[width=\linewidth]{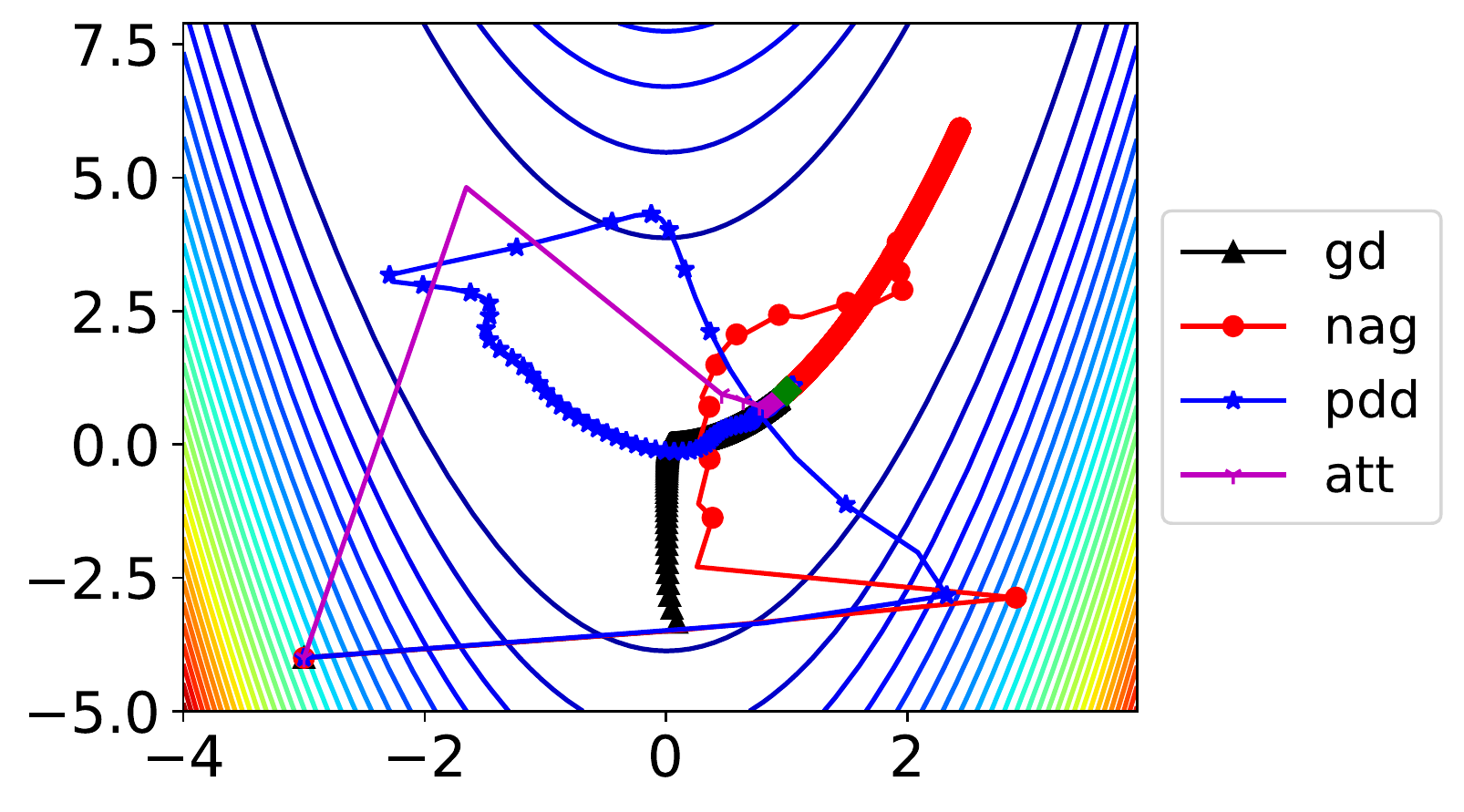}
  \caption{Optimization trajectories}
  \label{fig:rosen2}
\end{subfigure}
\caption{Minimizing the Rosenbrock function with gradient descent, NAG, PDD with $\mC(\vx) = \sI$ and IGAHD (`att'). The left panel shows the convergence speed of each method. The right panel shows the optimization trajectories of each method. }
\label{fig:rosenbrock}
\end{figure}
\subsubsection{N-dimension}
The $N$-dimensional coupled Rosenbrock function is defined as 
$$
f(\vx) = \sum_{i=1}^{N-1}\big((a-x_i)^2 + b(x_{i+1}-x_i^2)^2 \big)\,,
$$
 where we choose $a = 1$ and $b = 100$ as in the 2-dimensioal case and we set $N=100$. The global minimum is at $\vx^* = (1,1,\ldots,1)$.  Using the same stepsizes as in the 2-dimensional case, we compare the performance of the three algorithms starting from $\vx_0 = (0,\ldots,0)$. The stepsize of gradient descent is $\tau_{\textrm{gd}} = 0.001$. The stepsize of NAG is $\tau_{\textrm{nag}} = 0.0008$, $\beta = 0.95$. The parameters of PDD are $\tau_{\textrm{pdd}} = \sigma_{\textrm{pdd}}= 0.01$, $\varepsilon = 0.5$, $\omega = 1$, $A = 5$. The stepsize of the PDD method is larger than $\tau_{\textrm{gd}}$ and $\tau_{\textrm{nag}}$ because gradient descent and NAG do not allow larger stepsizes for convergence. For IGAHD (`att'), we choose $\tau_{\textrm{att}} = 0.0002$, $\alpha = 3$, $\beta^{(1)} = 2* \sqrt{\tau_{\textrm{att}}}$. The result is summarized in Fig.~\ref{fig:rosen_nd}
\begin{figure}
    \centering
    \includegraphics[width=0.5\textwidth]{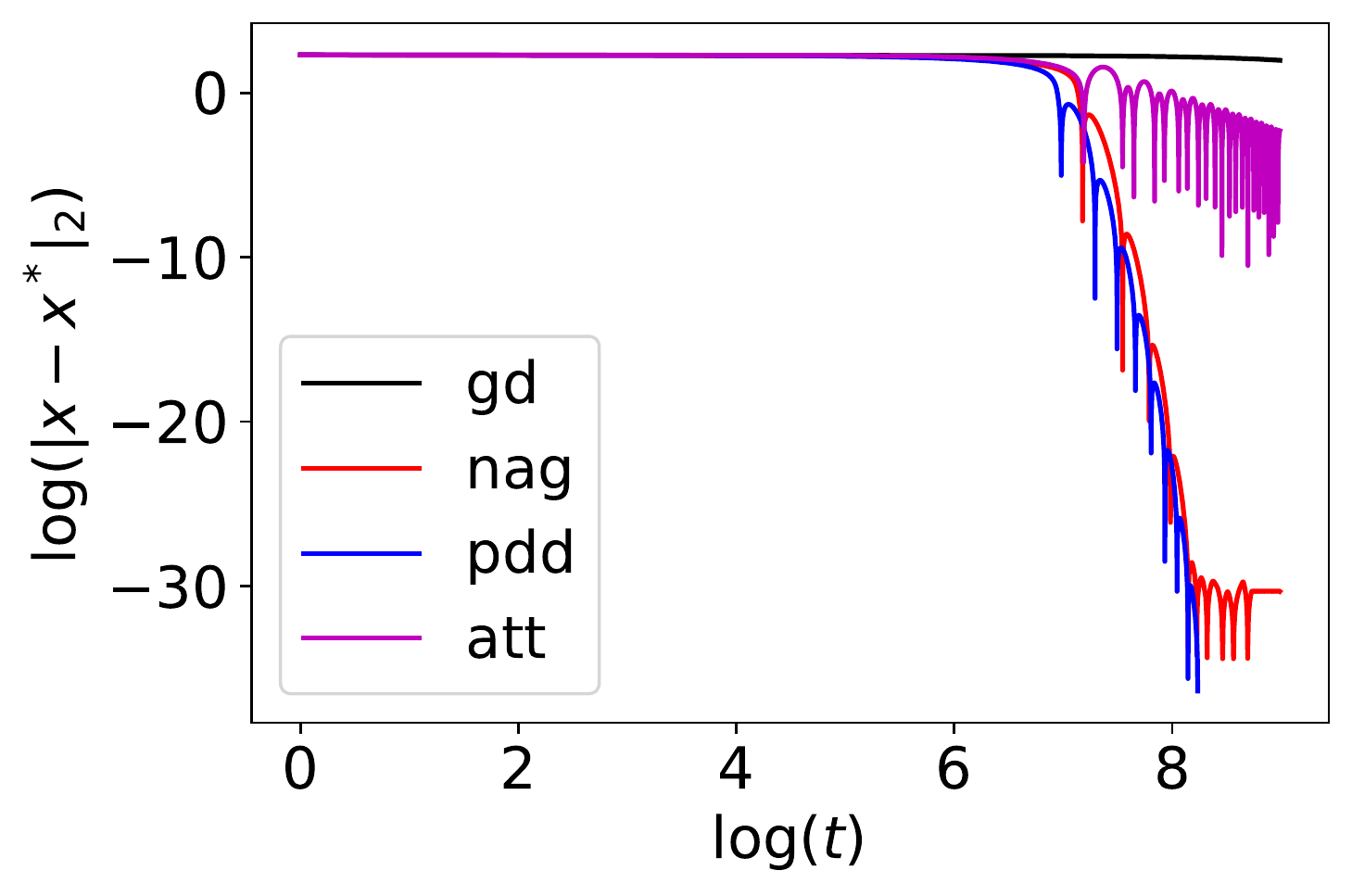}
    \caption{Comparison of gradient descent, NAG, PDD with $\mC(\vx) = \sI$ and IGAHD (`att') on minimizing the 100-dimensional coupled Rosenbrock function. The y-axis represents the distance between the current iterate and the global minimum on a logarithmic scale. The x-axis represents the number of iterations on a logarithmic scale. }
    \label{fig:rosen_nd}
\end{figure}
\subsection{Ackley function}
We consider minimizing the two-dimensional Ackley function given by 
$$ f(x,y) = -20 \exp\big(-0.2\sqrt{0.5(x^2 + y^2)}\big) - \exp \big[0.5 \big(\cos(2\pi x) + \cos(2\pi y)\big)\big] + \mathrm{e} + 20 \,,
$$
which has many local minima. The unique global minimum is located at $(x^*,y^*) = (0,0)$. We compare the performance of gradient descent, NAG, PDD, and IGAHD (`att') for minimizing the two-dimensional Ackley function starting from $(x_0,y_0) = (2.5,4)$. The stepsize of gradient descent is $\tau_{\textrm{gd}} = 0.002$. The stepsize of NAG is $\tau_{\textrm{nag}} = 0.002$, $\beta_{\textrm{nag}} = 0.9$. The parameters of PDD are $\tau_{\textrm{pdd}} = \sigma_{\textrm{pdd}}= 0.002$, $\varepsilon = 1$, $\omega = 1$, $A = 1$. For IGAHD (`att'), we choose $\tau_{\textrm{att}} = 0.01$, $\alpha = 3$, $\beta^{(1)} = 2* \sqrt{\tau_{\textrm{att}}}$. The results are summarized in Fig.~\ref{fig:ackley}. 
\begin{figure}
\centering
\begin{subfigure}{.5\textwidth}
  \centering
  \includegraphics[width=\linewidth]{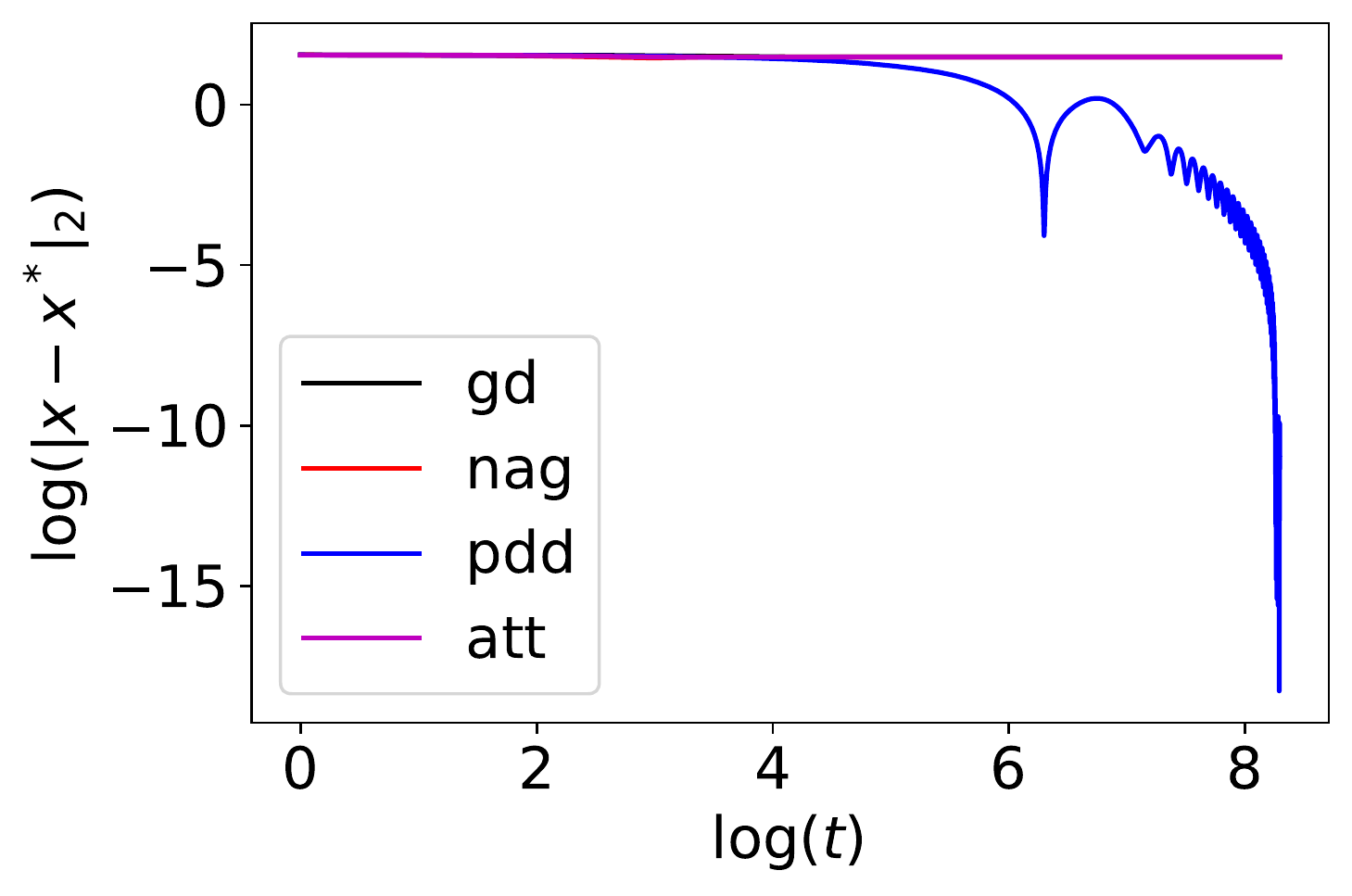}
  \caption{Convergence comparison}
  \label{fig:sub1}
\end{subfigure}%
\begin{subfigure}{.5\textwidth}
  \centering
  \includegraphics[width=\linewidth]{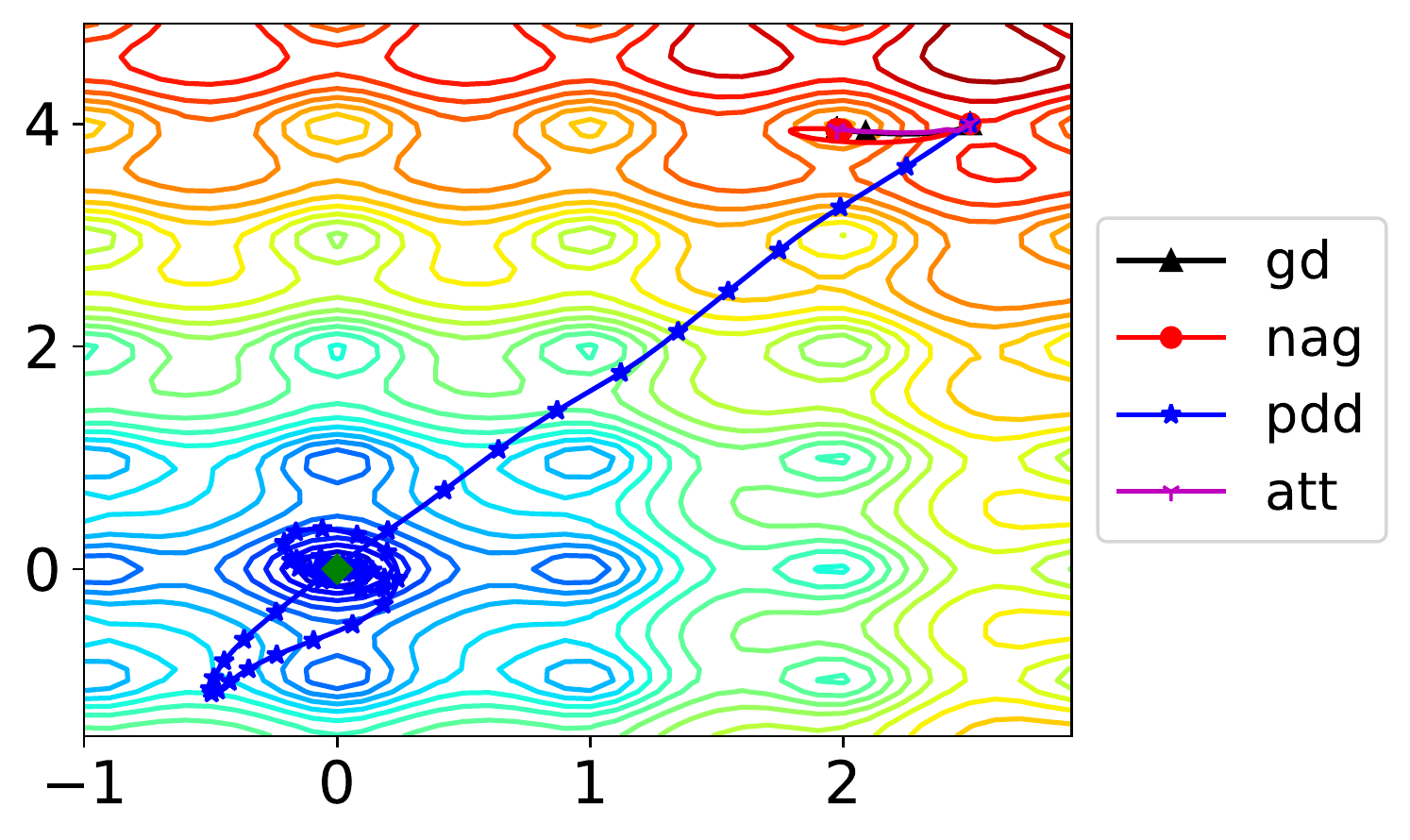}
  \caption{Optimization trajectories}
  \label{fig:sub2}
\end{subfigure}
\caption{Minimizing the Ackley function with gradient descent, NAG, PDD and IGAHD (`att'). The left panel shows the convergence speed of each method. The right panel shows the optimization trajectories of each method. }
\label{fig:ackley}
\end{figure}

\begin{remark}
    We remark that our algorithm has no stochasticity. It will not always converge to the global minimum for non-convex functions in general. For example, it will not converge for the Griewank, Drop-Wave, and Rastrigin functions. 
\end{remark}

\subsection{Neural Networks training}
\subsubsection{MNIST with Two-layer neural network}
We consider the classification problem using the MNIST handwritten digit data set with a two-layer neural network. The neural network has an input layer of size $784=28\times 28$, a hidden layer of size $32$ followed by another hidden layer of size $32$, and an output layer of size $10$. We use ReLU activation function across the layers, and the loss is evaluated using the cross-entropy loss. We use a batch size of 200 for all the algorithms. The stepsize of gradient descent is $\tau_{\textrm{gd}} = 0.001$. The stepsize of NAG is $\tau_{\textrm{nag}} = 0.001$, $\textrm{momentum} = 0.9$. The parameters of PDD are $\tau_{\textrm{pdd}} =  0.001$, $\sigma_{\textrm{pdd}}=5$, $\varepsilon = 0.005$, $\omega = 1$, $A = 1$. For IGAHD (`att'), we choose $\tau_{\textrm{att}} = 0.001$, $\alpha = 3$, $\beta^{(1)} = 0.01$. For Adam, we choose $\tau_{\textrm{adam}}=0.001$, $\beta_1 = 0.9$, $\beta_2=0.999$. 

\begin{table}[]
\centering
\begin{tabular}{|l|l|l|l|l|l|}
\hline
Algorithm & SGD                & NAG               & PDD &Adam &Att          \\ \hline
train loss  & 2.223 $\pm$ 0.034  & 0.964 $\pm$ 0.244  & \textbf{0.433 $\pm$ 0.270} &0.589 $\pm$ 0.282 & 0.591 $\pm$ 0.288 \\ \hline
test acc & 29.3 $\pm$ 8.3 \%  & 71.2 $\pm$ 9.4 \%  & \textbf{85.4 $\pm$ 10.4} \% &79.1 $\pm$ 11.3 \% &80.8 $\pm$ 11.4 \%  \\ \hline
\end{tabular}
\caption{Average training loss and test accuracy of different algorithms for MNIST handwritten digit recognition over 60 random seeds.}\label{table:mnist}
\end{table}

\begin{figure}
    \centering
    \includegraphics[width=\linewidth]{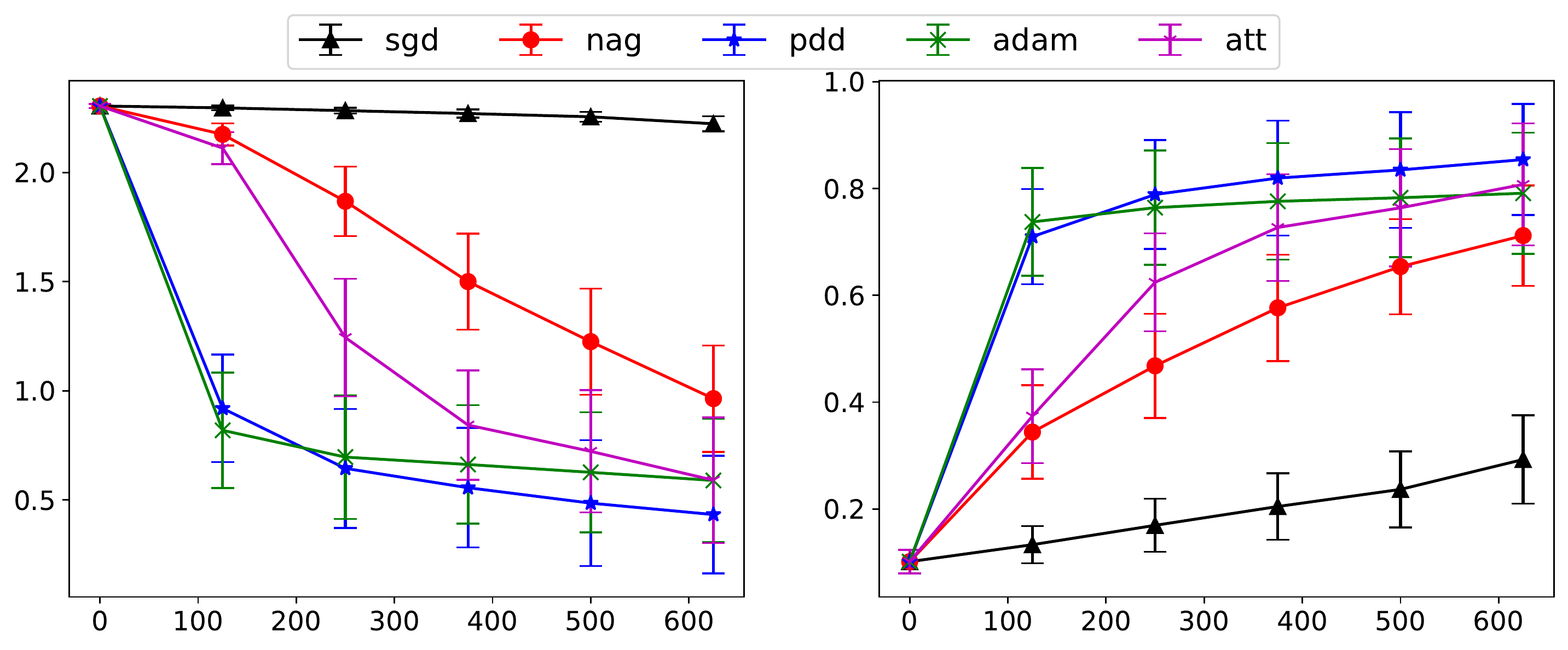}
    \caption{Training a two-layer neural network with the MNIST data set using gradient descent, NAG, PDD, Adam, and IGAHD (`att'). The left panel shows the convergence speed of training loss. The right panel shows the test accuracy of each method. The x-axis represents the number of iterations in terms of mini-batches.}
    \label{fig:MNIST}
\end{figure}

\subsubsection{CIFAR10 with CNN}
We train a convolutional neural network using the CIFAR10 datasets with SGD, Nesterov, PDD, Adam, and IGAHD (`Att'). The architecture of the network is described as follows. The network consists of two convolutional layers. The first convolutional layer has $32$ output channels, and the filter size is $3 \times 3$. The second convolutional layer has $64$ output channels, and the filter size is $4 \times 4$. Each convolutional layer is followed by a ReLU activation and then a $2 \times 2$ max-pooling layer. Lastly, we have $3$ fully connected layers of size $(64\cdot 4\cdot 4,120)$, $(120,84)$, and $(84,10)$. The loss is evaluated using the cross-entropy loss. The stepsize of gradient descent is $\tau_{\textrm{gd}} = 0.01$. The stepsize of NAG is $\tau_{\textrm{nag}} = 0.005$, $\textrm{momentum} = 0.9$. The parameters of PDD are $\tau_{\textrm{pdd}} =  0.005$, $\sigma_{\textrm{pdd}}=5$, $\varepsilon = 0.005$, $\omega = 1$, $A = 1$. For IGAHD (`att'), we choose $\tau_{\textrm{att}} = 0.005$, $\alpha = 3$, $\beta^{(1)} = 0.01$. For Adam, we choose $\tau_{\textrm{adam}}=0.005$, $\beta_1 = 0.9$, $\beta_2=0.999$.
\begin{figure}
    \centering
    \includegraphics[width=\linewidth]{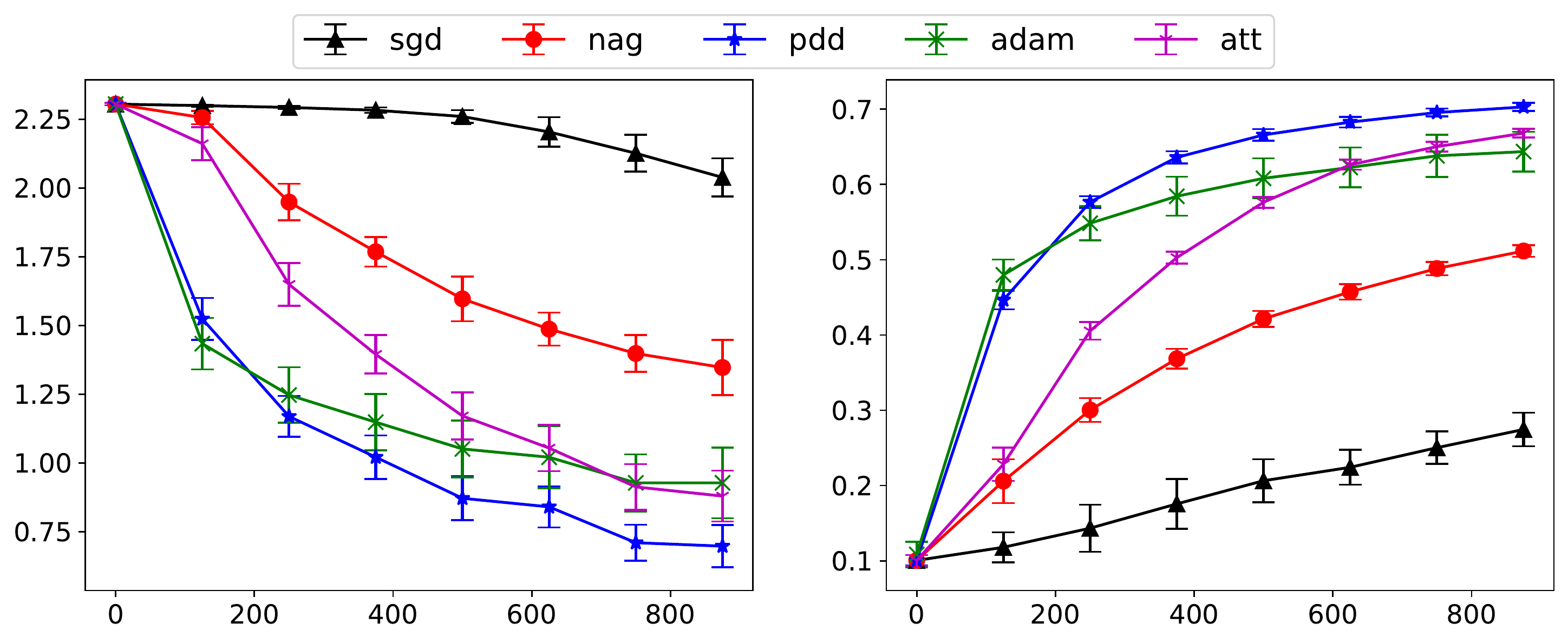}
    \caption{Training loss (left panel) and test accuracy (right panel) of a convolutional neural network on the CIFAR10 data set. The x-axis represents the number of iterations in terms of mini-batches. }
    \label{fig:CIFAR10}
\end{figure}
\begin{table}[]
\centering
\begin{tabular}{|l|l|l|l|l|l|}
\hline
Algorithm & SGD                & NAG               & PDD &Adam &Att          \\ \hline
train loss  & 2.038 $\pm$ 0.070 & 1.347 $\pm$ 0.100  & \textbf{0.697 $\pm$ 0.077} &0.927 $\pm$ 0.128 & 0.879 $\pm$ 0.092 \\ \hline
test acc & 27.5 $\pm$ 2.2 \%  & 51.2 $\pm$ 0.7 \%  & \textbf{70.3 $\pm$ 0.5} \% &64.4 $\pm$ 2.7 \% &66.8 $\pm$ 0.6 \%  \\ \hline
\end{tabular}
\caption{Average training loss and test accuracy of different algorithms for CIFAR10 data set over 60 random seeds.}\label{table:CIFAR10}
\end{table}




\section{Discussion}\label{section:discussion}
This paper presents primal-dual hybrid gradient algorithms for solving unconstrained optimization problems. We reformulate the optimality condition of the optimization problem as a saddle-point problem and then compute the proposed saddle-point problem by a preconditioned PDHG method. We present the geometric convergence analysis for the strongly convex objective functions. In numerical experiments, we demonstrate that the proposed method works efficiently in non-convex optimization problems, at least in some examples, such as  Rosenbrock and Ackley functions. In particular, it could efficiently train two-layer and convolution neural networks in supervised learning problems. 

So far, our convergence study is limited to strongly convex objective functions, not convex ones. Meanwhile, the choice of preconditioners and stepsizes are independent of time. We also have not discussed the optimal choices of parameters or general proximal operators in the updates of algorithms. These generalized choices of functions, parameters, and their convergence properties have been intensively studied in Nesterov accelerated gradient methods and Attouch's Hessian-driven damping methods. In future work, we shall explore the convergence property of PDHG methods for convex functions with time-dependent parameters. We also investigate the convergence of similar algorithms in scientific computing problems of implicit time updates of partial differential equations \citep{li2022controlling,li2023controlling,liu2023primal}. 

\noindent\textbf{Acknowledgement:} X. Zuo and S. Osher's work was partly supported by AFOSR MURI FP 9550-18-1-502 and ONR
grants:  N00014-20-1-2093 and N00014-20-1-2787. W. Li's work was supported by AFOSR MURI FP 9550-18-1-502, AFOSR YIP award No. FA9550-23-1-0087, and NSF RTG: 2038080.

\bibliography{main}
\newpage
\begin{appendices}
\section{Matrix lemma}
\begin{lemma}\label{lemma:matrix_CS}
Let $\mA, \mB, \mC \in \RR^n$ be real symmetric matrices that are simultaneously diagonalizable. Then for any $\vx,\vy \in \RR^n$, if 
$$
\lambda_{A,i} + \frac{|\lambda_{C,i}|}{2}\leq 0 \,
$$
$$
\lambda_{B,i} + \frac{|\lambda_{C,i}|}{2}\leq 0 \,
$$
for all $i$, where $\lambda_{A,i},\lambda_{B,i},\lambda_{C,i}$ are the $i$th eigenvalues of $\mA,\mB,\mC$ respectively in the same basis. Then 
$$
\vx^T \mA \vx + \vy^T \mB \vy + \vx^T \mC \vy \leq 0,
$$
for all $\vx,\vy \in \RR^n$. 
\end{lemma}
\begin{proof}
    Let $\vx,\vy \in \RR^n$. By our assumption, there exists $\mQ$ unitary such that $\mA,\mB,\mC$ are simultaneously diagonalizable by $\mQ$. Set $\tilde{\vx} = \mQ \vx$ and $\tilde{\vy} = \mQ \vy$. Then we can compute 
    \begin{align*}
        \vx^T \mA \vx + \vy^T \mB \vy + \vx^T \mC \vy &= \sum_{i=1}^n \tilde{x}_i^2 \lambda_{A,i} + \tilde{y}_i^2 \lambda_{B,i} + \tilde{x}_i \lambda_{C,i} \tilde{y}_i \\ 
        &\leq \sum_{i=1}^n \tilde{x}_i^2 \big(\lambda_{A,i}+\frac{|\lambda_{C,i}|}{2}\big) + \tilde{y}_i^2\big( \lambda_{B,i}+\frac{|\lambda_{C,i}|}{2}\big)\\
        &\leq 0 \,,
    \end{align*}
    where the first inequality follows from $\alpha x y \leq (x^2+y^2)|\alpha|/2$ for any $\alpha,x,y\in\RR$. 
\end{proof}
\section{Proof of Theorem \ref{prop:convergence_rate_original_pdd}}\label{appendix:proof_proposition_rate}
\subsection{Part (a)}
We have the following system of ODE: 
\begin{equation}\label{eq:ode_matrix}
    \begin{pmatrix}
        \dot{\vx} \\ \dot{\vp} 
    \end{pmatrix} = \begin{pmatrix}
        -\gamma \mB \mQ\mA\mQ &-\mB\mQ(\sI-\gamma \varepsilon \mA ) \\
         \mA\mQ & -\varepsilon \mA 
    \end{pmatrix} \begin{pmatrix}
        \vx \\ \vp 
    \end{pmatrix}\,.
\end{equation}
Let us compute the eigenvalues of the above system. Let $\alpha$ be an eigenvalue, then $\alpha$ satisfies 
\begin{align*}
    \mathrm{det} \begin{pmatrix}
          -\gamma \mB \mQ\mA\mQ - \alpha \sI &-\mB\mQ(\sI-\gamma \varepsilon \mA ) \\
         \mA\mQ & -\varepsilon \mA -\alpha \sI 
    \end{pmatrix} &= 0 \\
    \mathrm{det}\big( (-\gamma \mB \mQ\mA\mQ - \alpha \sI)(-\varepsilon \mA -\alpha \sI ) + \mB\mQ(\sI-\gamma \varepsilon \mA) \mA\mQ \big) &= 0\\
    \mathrm{det}\big( \alpha^2 \sI + \alpha (\varepsilon \mA + \gamma \mB \mQ \mA \mQ) + \gamma\varepsilon\mB\mQ\mA\mQ\mA + \mB\mQ\mA\mQ - \gamma \varepsilon \mB\mQ\mA\mA \mQ  \big) &= 0 \\
    \mathrm{det}\big( \alpha^2 \sI + \alpha (\varepsilon \mA + \gamma \mB \mQ \mA \mQ) + \mB\mQ\mA\mQ  \big) &= 0 \,.
\end{align*}
The late equality is because $\mA$ commutes with $\mQ$. 
We assume that $\mA$ and $\mB\mQ\mA\mQ$ are simultaneously diagonalizable. Thus, 
\begin{align*}
    0&=\alpha^2 + \alpha (\varepsilon a_i+ \gamma \mu_i) + \mu_i \,,  \\
    \alpha &= \frac{-\varepsilon a_i -\gamma\mu_i \pm \sqrt{(\varepsilon+\gamma \mu_i)^2 - 4\mu_i}}{2}\,.
\end{align*}
If $\gamma >0$ and $\varepsilon\geq 0$, then the real part of the eigenvalues are negative, and the system will converge. 
The convergence rate depends on the largest real part of the eigenvalues, which is 
$$
\max_i \frac{1}{2}\big[-\gamma \mu_i - \varepsilon a_i + \Re\big(\sqrt{(\gamma\mu_i + \varepsilon)^2-4\mu_i} \big)   \big]\,.
$$
\subsection{Part (c)}
When $\gamma = \varepsilon = 0$, we see that $\alpha$ is purely imaginary. Thus solutions to \eqref{eq:ode_matrix} will be oscillatory and will not converge. 

\subsection{Part (b)}
Let us define 
$$
g(\gamma) = \max_i\big\{ \frac{\mu_i\big(-\gamma + \Re\big(\sqrt{\gamma^2 - 4/\mu_i}\big)\big)}{2} \big\} \,.
$$
Essentially, we would like to find $\gamma^* = \argmin_{\gamma} g(\gamma)$. We then define  
\begin{align*}
    \gamma(\mu) &:= \argmin_{\gamma} \frac{\mu\big(-\gamma + \Re\big(\sqrt{\gamma^2 - 4/\mu}\big)\big)}{2} \\
    &= \frac{2}{\sqrt{\mu}} \,.
\end{align*}
Observe that if $\gamma \geq 2/\sqrt{\mu_n}$, then $\gamma^2-4/\mu_i \geq 0$ for all $i$. Thus 
\begin{align*}
    g(\gamma) &= \max_i\big\{ \frac{\mu_i\big(-\gamma + \sqrt{\gamma^2 - 4/\mu_i}\big)}{2} \big\}\,.
\end{align*}
For $\mu \in[\mu_n,\mu_1]$ and $\gamma \geq 2/\sqrt{\mu_n}$, one can check that the function $\mu\big(-\gamma + \sqrt{\gamma^2 - 4/\mu}\big)$ is increasing in $\mu$ by computing the partial derivative with respect to $\mu$. Then we get 
$$
g(\gamma) = \frac{\mu_1\big(-\gamma + \sqrt{\gamma^2 - 4/\mu_1}\big)}{2} \geq g(2/\sqrt{\mu_n})=\sqrt{\mu_1}(\sqrt{\kappa-1}-\sqrt{\kappa}) \approx -\sqrt{\mu_n}/2 \,,
$$
where $\kappa = \mu_1/\mu_n >1$. The last approximation is valid for $\mu_1/\mu_n \gg 1$. This shows that $\gamma^* \leq 2/\sqrt{\mu_n}$. Similarly, if $\gamma \leq 2/\sqrt{\mu_1}$, then $\gamma^2 - 4/\mu_i \leq 0$ for all $i$. Thus
\begin{align*}
    g(\gamma) &= \max_i\big\{ \frac{-\mu_i \gamma }{2} \big\} \\
    &= \frac{-\mu_n \gamma}{2} \\
    &\geq -\frac{\mu_n}{\sqrt{\mu_1}} = g(2/\sqrt{\mu_1})\,. 
\end{align*}
This shows that $\gamma^* \geq 2/\sqrt{\mu_1}$. Combining with our previous observation, we get $\gamma^* \in [2/\sqrt{\mu_1},2/\sqrt{\mu_n}]$. Now let us fix some $\gamma' \in [2/\sqrt{\mu_1},2/\sqrt{\mu_n}]$. Let $j = \inf \{i: 1\leq i\leq n, \gamma'^2 -4/\mu_i \leq 0 \}$. By our assumption on $\gamma'$, we know that $1<j<n$. Now for $1\leq i \leq j-1$, we have 
$$
\frac{\mu_i\big(-\gamma' + \Re\big(\sqrt{\gamma'^2 - 4/\mu_i}\big)\big)}{2} = \frac{\mu_i\big(-\gamma' + \sqrt{\gamma'^2 - 4/\mu_i}\big)}{2} \leq  \frac{\mu_1\big(-\gamma' + \sqrt{\gamma'^2 - 4/\mu_1}\big)}{2} \,.
$$
And for $j\leq k \leq n$, we have 
$$
\frac{\mu_k\big(-\gamma' + \Re\big(\sqrt{\gamma'^2 - 4/\mu_k}\big)\big)}{2} = \frac{-\mu_k\gamma' }{2} \leq  \frac{-\mu_n \gamma' }{2} \,.
$$
It is thus clear that for $\gamma'\in [2/\sqrt{\mu_1},2/\sqrt{\mu_n}]$, 
\begin{align*}
    g(\gamma') &= \max\big\{ \frac{\mu_1\big(-\gamma' + \sqrt{\gamma'^2 - 4/\mu_1}\big)}{2}, \frac{-\mu_n \gamma' }{2} \big\} \,.
\end{align*}
It is straightforward to calculate that for $\gamma \in [\frac{2}{\sqrt{\mu_1}},\frac{2\sqrt{\mu_1}}{\sqrt{\mu_n(2\mu_1-\mu_n)}}]$, we have 
$$
\frac{-\mu_n \gamma }{2} \geq \frac{\mu_1\big(-\gamma + \sqrt{\gamma^2 - 4/\mu_1}\big)}{2}\,.
$$
So 
$$
g(\gamma) = \frac{-\mu_n \gamma }{2} \geq g(\frac{2\sqrt{\mu_1}}{\sqrt{\mu_n(2\mu_1-\mu_n)}}) = \frac{-\sqrt{\mu_n}}{\sqrt{2-\frac{1}{\kappa}}}\,. 
$$
And for $\gamma \in [\frac{2\sqrt{\mu_1}}{\sqrt{\mu_n(2\mu_1-\mu_n)}},2/\sqrt{\mu_n}] $ we have 
$$
\frac{-\mu_n \gamma }{2} \leq \frac{\mu_1\big(-\gamma + \sqrt{\gamma^2 - 4/\mu_1}\big)}{2}\,.
$$
This implies 
$$
g(\gamma) = \frac{\mu_1\big(-\gamma + \sqrt{\gamma^2 - 4/\mu_1}\big)}{2} \geq g(\frac{2\sqrt{\mu_1}}{\sqrt{\mu_n(2\mu_1-\mu_n)}}) = \frac{-\sqrt{\mu_n}}{\sqrt{2-\frac{1}{\kappa}}}\,. 
$$
This shows $\gamma^* = \frac{2\sqrt{\mu_1}}{\sqrt{\mu_n(2\mu_1-\mu_n)}}$. 
\subsection{Part (d)}
Define $\Delta_{\gamma}(\mu,\varepsilon) = (\gamma \mu+\varepsilon)^2-4\mu$. Also define $g_{\gamma}(\mu)= 2\sqrt{\mu}-\gamma \mu$. Then for $\mu\geq 0$, we have $\Delta_{\gamma}(\mu,\varepsilon) \leq 0 $ if and only if $\varepsilon \leq g_{\gamma}(\mu)$. Note that $g_{\gamma}'(\mu) = \frac{1}{\sqrt{\mu}}-\gamma \geq 0 $ for $\mu \leq \mu_1$ if $\gamma \leq \frac{1}{\sqrt{\mu_1}}$. Then $\Delta_{\gamma}(\mu,\varepsilon) \leq 0$ for all $\mu \leq \mu_1$ if $\gamma \leq \frac{1}{\sqrt{\mu_1}}$ and $\varepsilon \leq g_{\gamma}(\mu_n)$. In particular, $\Delta_{\gamma}(\mu,\varepsilon) \leq 0$ for all $\mu \leq \mu_1$ if $\varepsilon = g_{\gamma}(\mu') $ for some $\mu' \leq \mu_n$. We have  
\begin{align*}
    \alpha&=\max_i \frac{1}{2}\big[-\gamma \mu_i - \varepsilon + \Re\big(\sqrt{(\gamma\mu_i + \varepsilon)^2-4\mu_i} \big)   \big] \\
    &= \max_i \frac{1}{2}\big[-\gamma \mu_i - \varepsilon  \big] \\
    &= \max_i \frac{1}{2}\big[-\gamma \mu_i - 2 \sqrt{\mu'}+\gamma \mu'  \big] \\
    &= -\sqrt{\mu'} - \frac{\gamma(\mu_n - \mu')}{2}.
\end{align*}

\section{Proof of Proposition \ref{prop:pdd_ode_regu} }\label{appendix:proof_pdd_ode_regu}
We directly compute 
\begin{align} 
    \ddot{\vx} &= -\mC\big((\sI-\gamma \varepsilon \mA) \dot{\vp} + \gamma \mA \nabla^2 f(\vx) \dot{\vx}\big) -\dot{\mC}\big((\sI-\gamma \varepsilon \mA )\vp + \gamma \mA\nabla f(\vx) \big)   \nonumber \\
    &= -\mC\big((\sI-\gamma \varepsilon \mA ) (\mA \nabla f(\vx) - \varepsilon \mA   \vp)   + \gamma \mA \nabla^2 f(\vx) \dot{\vx}\big) \nonumber \\
    &\qquad-\dot{\mC}\big((\sI-\gamma \varepsilon \mA  )\vp + \gamma \mA\nabla f(\vx) \big)   \nonumber \\
    &=-\mC\big[(\sI-\gamma \varepsilon \mA  ) \mA \nabla f(\vx) + \varepsilon \mA   (\mC^{-1}\dot{\vx} + \gamma \mA \nabla f(\vx))  + \gamma \mA\nabla^2 f(\vx) \dot{\vx} \big] +\dot{\mC} \mC^{-1}\dot{\vx} \nonumber \\
    &= -\mC\big[\mA \nabla f(\vx) + \varepsilon \mA  \mC^{-1}\dot{\vx}  + \gamma \mA\nabla^2 f(\vx) \dot{\vx}  \big] +\dot{\mC} \mC^{-1}\dot{\vx} \,.\nonumber 
\end{align}

\end{appendices}

\end{document}